%% file: comp_resolve_revised.tex
\documentclass[a4paper,10pt]{amsart}

\input{header.tex}

\begin{document}

\title[Computing resolutions of quotient singularities]{Computing resolutions of quotient singularities}
\author[M.~Donten-Bury]{Maria~Donten-Bury}
\address{University of Warsaw, Institute of Mathematics, Banacha 2, 02-097 Warszawa, Poland \& Freie Universit\"at Berlin, Mathematisches Institut, Arnimallee 3, 14195 Berlin, Germany}
\email{m.donten@mimuw.edu.pl}
\author[S.~Keicher]{Simon~Keicher}
\address{
Departamento de Matematica, 
Facultad de Ciencias Fisicas y Matematicas,
Universidad de Concepci{\'o}n, Casilla 160-C, Concepci{\'o}n, Chile
}
\email{keicher@mail.mathematik.uni-tuebingen.de}
\thanks{This work was completed while the first author held a Dahlem Research School Postdoctoral Fellowship at Freie Universit\"at Berlin.
The first author was partially supported by a Polish National Science Center project 2013/11/D/ST1/02580 and the second author was supported by Proyecto FONDECYT Postdoctorado N.~3160016.}

\subjclass[2010]{14E15, 14Q10, 14Q15, 14C20, 14L24}

\begin{abstract}
Let $G\subseteq\GL(n)$ be a finite group without pseudo-reflections. 
We present an algorithm to compute and verify a candidate for the Cox ring of a resolution $X\to \KK^n/G$, which is based just on the geometry of the singularity $\KK^n/G$, without further knowledge of its resolutions.
We explain the use of our implementation of the algorithms in  \texttt{Singular}.
As an application, we determine the Cox rings of resolutions $X\to \KK^3/G$ for all $G\subseteq \GL(3)$ with the aforementioned property and of order~$|G|\leq 12$. We also provide examples in dimension~$4$.
\end{abstract}
\maketitle

\section{Introduction}

We consider quotient spaces $\KK^n/G = \spec\: \KK [x_1,\ldots,x_n]^G$, where $G$ is a finite group acting linearly on~$\KK^n$. By Noether's theorem the ring of invariants of~$G$ is finitely generated, hence such quotients are (complex) algebraic varieties. They are usually singular and in such a case are called \emph{quotient singularities}. In their construction, geometry meets finite group theory and there have been many attempts of extending this relation to resolutions of $\KK^n/G$. An example, probably the most important one, of describing the geometric structure of crepant (which in this case means that the canonical divisor is linearly trivial) resolutions of $\KK^n/G$ in terms of algebraic properties of the group~$G$ is the McKay correspondence; for an introduction see, e.g.,~\cite{ReidBourbaki}. Though proved in several cases (see section~\ref{subsect:properties-resolutions}), in general it reveals how much there is still to learn about this class of singularities and their resolutions.

In this paper, we study quotient singularities $X_0 := \KK^n/G$ and their resolutions in terms of the {\em Cox ring}
$$
\Cox\left(X_0\right)
\ =\ 
\bigoplus_{\Cl(X)} \Gamma\left(X_0,\OO(D)\right),
$$
see~\cite{ArDeHaLa} for details on the construction. 
Cox rings have already been successfully used by various authors to study resolutions of quotient singularities, in particular for symplectic quotients in~\cite{81resolutions,DobuMa}, described by a generating set in a simpler ring, and via an algorithmic approach based on toric ambient modifications in~\cite{HaKe}.
In this article we generalize the methods used in~\cite{HaKe}.
While finishing a draft of this paper, we also found out about Yamagishi's work~\cite{Ya}, which extends~\cite{81resolutions, DobuMa}.
However, our approach is different from~\cite{Ya} and our methods are not restricted to the class of crepant resolutions. In particular, as explained below, we do not try to construct the generators of the Cox ring directly from the group structure data, but we obtain it via toric ambient modifications, chosen in an intermediate step based on tropical geometry.

Our first contribution, presented in Section~\ref{sec:algos}, is an an algorithm to compute and verify a candidate for the Cox ring of a resolution $X\to X_0$ without requiring further knowledge of $X_0$.
 More precisely, Algorithm~\ref{algo:coxquotsing} computes $\Cox(X_0)$
and  Algorithm~\ref{algo:resolve} then computes a candidate for the Cox ring of a resolution $X\to X_0$ and verifies the result.
Here, the main tool are toric ambient modifications as in~\cite{Ha2, HaKeLa}:
we embed $X_0$ into an affine toric variety $Z_0$ and compute a resolution $Z\to Z_0$. The proper transform $X\to X_0$
then is the desired candidate for a resolution which we can verify algorithmically.
Note that the choice of the toric resolution involves a tropical step; this is in the spirit of Tevelev and Teissier~\cite{Tev,Tev2}.
Algorithm~\ref{algo:resolve} is a variant of the algorithm~\cite{Ke:diss} given for Mori dream spaces, which, in turn, is based on~\cite{ArDeHaLa,Hu:diss} where the algorithm has been shown to work in the setting of complete rational complexity one $T$-varieties. See also~\cite{HaWro} for the case of affine $\KK^*$-surfaces.
Our algorithms are implemented in a library for the Open Source computer algebra system \texttt{Singular}~\cite{singular}; we explain its use in Section~\ref{sec:implementation} by examples.

Our second contribution
concerns resolutions of three-dimensional quotient singularities;
it is presented in  Section~\ref{sec:quotsing}.
We first classify in Proposition~\ref{prop:reps} the finite, non-abelian groups $G\subseteq \GL(3)$ with $|G|\leq 12$ and
without pseudo-reflections (see Remark~\ref{rem:pseudo-reflections}). When~$G$ is abelian, the quotient singularity is toric; since the algorithm of constructing a (toric) resolution and the structure of the Cox ring is known in this case, we do not consider it.
Using our algorithms from Section~\ref{sec:algos}, we first present the Cox rings of all singularities~$\KK^3/G$ on the list, see Proposition~\ref{prop:coxX0} and then their resolutions, see Theorem~\ref{thm:dim3}.
Then we discuss certain properties of the obtained resolutions in order to understand what can be expected from the output of the algorithm in general. In particular, we check whether the resolutions for subgroups of $\SL(3)$ on the list are crepant.
In Subsection~\ref{subsection:smallres}, we then give a modified algorithm to produce smaller resolutions. We apply it to obtain two more crepant resolutions.

Finally, in Section~\ref{sec:4dimex}, we apply our methods to two four dimensional examples.
Note that in dimension $4$ less is known about resolutions of quotient singularities, hence computational experiments are even more valuable than in dimension~$3$.
We provide the Cox rings of $X_0=\KK^4/G$ for chosen representations of $G=S_3$ and $G=D_8$ and compute the Cox rings of modifications $X\to X_0$. 
We are able to retrieve crepant resolutions that were also found in~\cite{DobuMa}, see Proposition~\ref{prop:4dim1} and~\ref{prop:4dimS3}.

{\em Acknowledgements: } We would like to thank J\"urgen Hausen for several helpful discussions.
We further thank Yue Ren and Christof S\"oger respectively for answering question about Singular and Normaliz.
Lastly, we would like to thank the anonymous referee for his valuable comments and suggestions.

\section{Algorithmic resolution}
\label{sec:algos}

In this section, we describe the algorithm to compute the Cox ring of a resolution of a quotient singularity $X_0:=\KK^n/G$ where $G$ is a finite group. It is divided into two parts: the first one, Algorithm~\ref{algo:coxquotsing}, returns a presentation of the Cox ring of $X_0$ and the second one, Algorithm~\ref{algo:resolve}, computes the Cox ring of a candidate for a resolution $X\to X_0$ and then verifies the choice of $X$. The latter algorithm is a variant of~\cite[Thm 3.4.4.9]{ArDeHaLa} and \cite[Alg.~2.4.8]{Ke:diss}. It produces such a candidate $X$ using a tropical step.
 
 \subsection{The setting and Cox rings}
We work with representations of~$G$, usually assuming that they are faithful. In this case we will often identify the representation with its image in $\GL(n)$, i.e., consider~$G$ as a matrix group.

\begin{definition}
 A matrix group $G\subseteq \GL(n)$ is \emph{small} if it does not have \emph{pseudo-reflections};
these are $A\in G\subseteq \GL(n)$ of finite order
such that the subspace of fixed points $(\KK^n)^A$ is a hyperplane.
We also say that a representation of an abstract group $G$ is small, if its image in $\GL(n)$ is small.
\end{definition}

\begin{remark}
\label{rem:pseudo-reflections}
Note that the assumption on the representation of $G$ not having 
pseudo-reflections is not restrictive:
let $H\subseteq G$ be the normal subgroup generated by all pseudo-reflections. 
Then
$$
\KK^n/G\ \ \cong\ \ (\KK^n/H) \bigl/ (G/H) \ \ \cong\ \ \KK^n\bigl/(G/H)
$$
since $\KK^n/H$ is smooth by the Chevalley-Shephard-Todd theorem.
In particular, the singularity for $G$ is the same as the one for the smaller group~$G/H$.
\end{remark}

As outlined in the introduction, 
our central tool for the study of both $X_0$ and its resolutions is the Cox ring.
In this paper, we are particularly interested in the case
where the Cox ring and the class group are finitely generated; one then 
calls the underlying variety a {\em Mori dream space}.
Given such a Mori dream space $X$, the Cox ring  $\Cox(X)$ together with a choice 
of a collection of overlapping polyhedral cones in $\Cl(X)\otimes_\ZZ \QQ$ determines $X$ up to isomorphism.
If $X$ is a surface, $\Cox(X)$ even fully encodes~$X$.
The following construction summarizes this situation and shows how $X$ can be retrieved as a quotient of an open subset of $\Spec(\Cox(X))$.
We refer to the book~\cite[Sect.~3]{ArDeHaLa} for details.

To this end, recall that a variety~$X$
is an {\em $A_2$-variety} if each two points admit a common affine open neighborhood. Examples include the class of quasiprojective varieties.
We call $X$ an {\em $A_2$-maximal} variety if $X$ is an $A_2$-variety 
and there is no open embedding $X\subsetneq Y$
into an $A_2$-variety $Y$ and $\codim_Y(Y\setminus X)\geq 2$

\begin{construction}
\label{con:bunchedring}
(See~\cite[Sect.~3I]{ArDeHaLa} and~\cite{Ha2}.)
Let $X$ be a normal, irreducible $A_2$-maximal variety with $\Gamma(X,\OO^*)=\KK^*$ such that both $K:=\Cl(X)$ and $R:=\Cox(X)$ are finitely generated. 
Then $X$ is the good quotient of an open subset~$\widehat X$ of~$\Spec(R)$ 
$$
\xymatrix{
X
&&
\ar[ll]^(.7){\good H}
\widehat X\ \subseteq\ \b X\ :=\ \Spec(R)\ \subseteq\ \KK^r
}
$$ 
by the action of the characteristic quasitorus 
$H:=\Spec(\KK[K])$, and  $\widehat X$ can be described combinatorially by a set $\Phi$ of pairwise overlapping polyhedral cones in $K\otimes_\ZZ\QQ$.
The triple $(R,\mathfrak F, \Phi)$, where $\mathfrak F$ is a generating set of $R$, is called a {\em bunched ring} and already determines $X$ up to isomorphism.
\end{construction}

Moreover, each variety~$X$ as in Construction~\ref{con:bunchedring}
comes with an embedding
into the so-called {\em canonical  toric ambient variety};
this is a toric variety $Z\supseteq X$ with a fan $\Sigma \subset N\otimes_\ZZ \QQ$ 
such that the embedded variety $X$ inherits several nice properties from $Z$; we refer to~\cite[Subsect.~3.2.5]{ArDeHaLa} for details as well as its explicit computation from the Cox ring of~$X$.

\subsection{Cox ring of the quotient singularity}
We start from explaining how to compute a presentation $\Cox(X_0)=\KT{s}/I_0$ of the Cox ring of $X_0 = \KK^n/G$, which is suitable for being an input data for computing the Cox ring of a resolution of~$X_0$. The algorithm generalizes~\cite[Prop.~3.1]{HaKe}.
First note that, by~\cite[Thm~3.1]{ArGa}, $\Cox(X_0)$ is isomorphic to the ring of invariants $\KK[S_1,\ldots,S_n]^{[G,G]}$
that comes with a grading by the (finite) group 
$\Cl(X_0)= \XX(G')$ of characters of the abelianization $G':=G/[G,G]$. 
On the geometric side, the situation becomes (cf.~\cite[Prop.~3.1]{HaKe}):
\[
 \xymatrix{
 \CC^n/[G,G] 
 \ar@{}[rr]|\subseteq
 \ar[d]^{/G'}
 &&
 \CC^s
 \ar[d]_{/G'}
 \\
 X_0
\ar@{}[rr]|\subseteq
  &&
 \CC^s/G'
 }
\]

To obtain a presentation suitable for further steps, we need to find a set of $G'$-homogeneous generators $g_1,\ldots, g_s$ of $\KK[S_1,\ldots,S_n]^{[G,G]}$. Thus we first compute any minimal set of generators 
and then turn it into a set of $G'$-homogeneous ones.
This is done separately in the sets of polynomials of a fixed (standard) degree as follows.

\begin{construction}[$G'$-homogeneous generators for {$\KK[S_1,\ldots,S_n]^{[G,G]}$}]
\label{con:GGhomog}
Consider generators $\mathcal F = \{f_1,\ldots,f_s\}$ 
for $\CC[S_1,...,S_n]^{[G,G]}$.
 Let $d\in\ZZ_{\geq 0}$ be such that there is $f_i \in \mathcal F$ of degree~$d$. Let $V_d$ denote the vector space of all polynomials of degree~$d$ in $\KK[S_1,\ldots,S_n]^{[G,G]}$. Denote by $h_{d,1},\ldots,h_{d,s_d}$ a basis of~$V_d$ chosen such that its elements either belong to~$\mathcal F$ or are products of elements of~$\mathcal F$ of smaller degree. 

Observe that we have an induced action of $G'$ on~$V_d$.
Thus, there is a homomorphism (often, but not always, an injection)
$$
G'\ \to\ \GL(s_d),\qquad
G' \ni [m] \ \mapsto\ M_{d,m},
$$
where $M_{d,m}\in\GL(s_d)$ of the automorphism of $V_d$ given by $m\in G\subseteq \GL(n)$ in basis $h_{d,1},\ldots,h_{d,s_d}$.
We describe the image of $G'$ in $\GL(s_d)$ by computing a generating set
$\<M_{d,1},\ldots,M_{d,k}\>\subseteq \GL(s_d)$, which is an image of a generating set of~$G$.
Moreover, since $G'$ is abelian, we can find a basis $\overline{h}_{d,1},\ldots,\overline{h}_{d,s_d}$ of $V_d$ in which all $M_{d,i}$ are simultaneously diagonalized.

Going through the possible values of $d\in \ZZ_{\geq 0}$,
we collect the bases $\overline{h}_{d,1},\ldots,\overline{h}_{d,s_d}$
and remove redundant elements; 
the result is a set of $G'$-homogeneous polynomials $g_1,\ldots,g_s$ which generate $\CC[S_1,\ldots,S_n]^{[G,G]}$.
We will write $\chi_{ij}$ for the eigenvalue of $g_i$ with respect to $M_{d_i,j}$, where $d_i = \deg(g_i)$.

\end{construction}

Let $g_1,\ldots,g_s$ and $\chi_{ij}$ be as in Construction~\ref{con:GGhomog}.
We then obtain the desired presentation $\Cox(X_0)=\KT{s}/I_0$ by computing the kernel $I_0\subseteq \KT{s}$ of the homomorphism
$$
\KT{s}\,\to\,\KK[S_1,\ldots,S_n]^{[G,G]},\qquad
T_i\,\mapsto\,g_i.
$$
This is a standard procedure, it requires eliminating variables $S_1,\ldots,S_n$ from the ideal generated by all $T_i-g_i$ in the enlarged ring $\KK[S_1,\ldots,S_n,T_1,\ldots,T_r]$, see for example~\cite[Thm.~3.2]{CoLiOSh} or~\cite[Lem.~2.1]{ker}.
Finally, to install the $\XX(G')=\Cl(X_0)$-grading on $\Cox(X_0)$,
it remains to define the degrees of the generators as $\deg(T_i)=(\chi_{i1},\ldots,\chi_{ik})\in {\mathbb X}(G')$;
note that we give only values on a generating set of $G'$, which determine the character uniquely.
The {\em degree map} of $\Cox(X_0)=\KT{s}/I_0$, i.e.,
the homomorphism 
$$
\ZZ^s\,\to\,\Cl(X_0),\qquad
e_i\,\mapsto\,\deg(T_i).
$$
is defined by the {\em degree matrix}, the matrix with columns~$\deg(T_1),\ldots,\deg(T_s)$.
We summarize this as follows.

\begin{algorithm}[Cox ring of $X_0$]
\label{algo:coxquotsing}
\emph{Input: }
a small group $G\subseteq \GL(n)$, given by a generating set.
 \begin{itemize}
  \item 
  Compute generators $f_1,\ldots,f_s$ for the invariant ring 
  $\KK[S_1,\ldots,S_n]^{[G,G]}$.
  \item 
  As in Construction~\ref{con:GGhomog}, modify $f_1,\ldots,f_s$ to a set $g_1,\ldots,g_s$ of $G'$-homo\-gen\-eous generators of $\KK[S_1,\ldots,S_n]^{[G,G]}$ by computing eigenvectors of the induced action of $G'$ on subspaces of polynomials of fixed degree~$d$, where $d\in \{\deg(f_1),\ldots,\deg(f_s)\}$.
  \item 
  In $\KK[S_1,\ldots,S_n,T_1,\ldots,T_s]$, define 
  the ideal $J=\<T_i-g_i;\, 1\leq i\leq s\>$
  and   compute its elimination ideal
  $I_0:=J\cap \KK[T_1,\ldots,T_s]$.
  \item
  Define the matrix $Q_0$ with the $i$-th column $(\chi_{i1},\ldots,\chi_{ik})$, where $\chi_{ij}$ are eigenvalues of $g_i$ with respect
  to~$M_{d_i,j}$, as in Construction~\ref{con:GGhomog}.
 \end{itemize}
\emph{Output: }
$I_0\subseteq \KT{s}$ and $Q_0$. Then $\KT{s}/I_0$ is a presentation of $\Cox(X_0)$, where $X_0=\KK^n/G$, in terms of generators and relations with $Q_0$ as degree matrix.
\end{algorithm}

\begin{example}
\label{ex:Q8v1}
We consider the faithful representation of the quaternion group $G:=Q_8$ with $8$ elements
whose image $G\subseteq\GL(2)$ is generated by
\begin{align*}
   \bigl(\begin{smallmatrix}i&0\\0&-i\end{smallmatrix}\bigr),\qquad 
   \bigl(\begin{smallmatrix}0&-i\\-i&0\end{smallmatrix}\bigr),
  \end{align*}
where $i\in \CC$ denotes the imaginary unit.
There are no pseudo-reflections since no matrix in~$G$ has $1$ as eigenvalue.
We now apply the steps of Algorithm~\ref{algo:coxquotsing}.  
Generators for the invariant ring $\KK[S_1,S_2]^{[G,G]}$,
where $[G,G]=\< -\id\>$, are
$$
f_1\ =\ S_1S_2,\qquad
f_2\ =\ S_1^2,\qquad
f_3\ =\ S_2^2.
$$
Moreover, the group $G'=G/[G,G]$ is isomorphic to $(\ZZ/2\ZZ)^2$
and acts on the linear hull $V_2:=\Lin_\KK(f_1,f_2,f_3)$;
we work with the description
$$
G'\ \cong\ 
\left\<
\left[
    \mbox{\tiny $
    \begin{array}{rrr}
    1 & 0 & 0 \\
    0 & -1 & 0 \\
    0 & 0 & -1
    \end{array}
    $}
    \right],
\left[
    \mbox{\tiny $
    \begin{array}{rrr}
    -1 & 0 & 0 \\
    0 & 0 & 1 \\
    0 & 1 & 0
    \end{array}
    $}
    \right]
\right\>
\ \subseteq\ \GL(3).
$$
The vector space $V_2$ then has the $G'$-invariant subspaces
$\Lin_\KK(f_1)$ and $\Lin_\KK(f_2,f_3)$.
We have a basis $(g_1,g_2,g_3)$ for $V_2$ consisting of the eigenvectors
\begin{gather*}
g_1\ :=\ f_1,\qquad
g_2\ :=\ f_2+f_3,\qquad
g_3\ :=\ -f_2+f_3\ \in\ V
\end{gather*}
where the weights $w_i := (\chi_{i1},\chi_{i2})\in (\ZZ/2\ZZ)^2\cong G'$ are 
$w_1=(\b 0,\b 1)$, $w_2=(\b 1,\b 0)$ and $w_3=(\b 1,\b 1)$.
The degree map is
$$
Q_0\colon \ZZ^3\to (\ZZ/2\ZZ)^2,\qquad
e_i\mapsto 
\left[
\mbox{\footnotesize $
\begin{array}{rrr}
\b 0 &\b  1 &\b  1\\
\b 1 &\b  0 &\b  1
\end{array}
$}
\right]
\cdot e_i
$$
In the ring $\KK[S_1,S_2,T_1,T_2,T_3]$, we consider the ideal
$J:=\<T_1 - g_1, T_2 - g_2, T_3-g_3\>$
  and compute its elimination ideal
  $I_0:=J\cap \KK[T_1,T_2,T_3]$.
  We arrive at the $G'$-graded Cox ring with degree matrix~$Q_0$:
  $$
  \Cox(\KK^2/G)
  \ \cong\ 
  \KK[T_1,T_2,T_3]/I_0,\qquad  
  I_0\, =\,
  \<4T_1^2 - T_2^2 + T_3^2\>.
  $$
\qed
\end{example}

\subsection{Cox ring of the resolution}
We now turn to the computation of the Cox ring of a resolution using the presentation $\Cox(X_0)=\KT{s}/I_0$ of the Cox ring 
 obtained using Algorithm~\ref{algo:coxquotsing}. There are two basic ingredients of the algorithm: the tropical variety $\Trop(X_0)$ constructed from the presentation of~$\Cox(X_0)$ and the algorithm for computing the Cox ring of a toric ambient modification,~\cite[Alg.~3.6]{HaKeLa}. The idea is to embed $X_0$ into its ambient affine toric variety $Z_0 = \KK^n/G'$ and to use the tropical variety $\Trop(X_0)$ to determine the centers for toric ambient modifications of $X_0$. Then we perform a toric resolution $Z\to Z_0$, consider the proper transform $X\to X_0$, compute its Cox ring and check whether it actually corresponds to a resolution of~$X_0$. Let us explain the details.

\smallskip
\textit{Step 1: the toric ambient variety.}
Recall that $Q_0$ is the degree matrix of $\Cox(X_0)$. We fix a {\em Gale dual} matrix for $Q_0$, i.e., a matrix $P_0$ such that its transpose $P_0^*$ fits into the exact sequence
\begin{align}
\xymatrix{
0
&
\ar[l]
\Cl(X_0)
&&
\ar[ll]_{
Q_0
}
\ZZ^s
&&
\ar[ll]_{P_0^*}
\ker(Q_0)
&
\ar[l]
0
}
\label{eq:seq}
\end{align}
see, e.g.,~\cite{Ke:diss} for its computation.
The defining fan $\Sigma_0$ of the canonical toric ambient variety $Z_0$
of $X_0$ (cf.~the paragraph after Construction~\ref{con:bunchedring}) can be written down explicitly:
$$
\Sigma_0\ =\ {\rm faces}\left(\sigma\right),\qquad
\sigma\ :=\ P_0(\QQQ^s)\ \subseteq\ \QQ^n.
$$
The corresponding (affine) toric ambient variety is $Z_0=\KK^n/G'$ with the abelianization $G':=G/[G,G]$.

\smallskip
\textit{Step 2: the tropical variety.}
The tropical variety~$\Trop(X_0)$ lies in the same linear space 
$\QQ^n$ as the fan~$\Sigma_0$.
Consider the affine variety $\b X_0 := \Spec(\Cox(X_0))\subseteq \KK^s$.

\begin{construction}\label{con:trop}
Given a polynomial $f\in \KT{s}$, its tropical variety $\Trop(f)\subseteq \QQ^s$ is the support of the codimension-one skeleton of the normal fan over the Newton polytope of~$f$.
In the previous setting, we then obtain tropical varieties 
\begin{gather*}
 \Trop\left(\b X_0\cap (\KK^*)^s\right)
 \ :=\, 
 \bigcap_{f\in I_0} \!\Trop(f)
 \ \subseteq\ \QQ^s,
\\
\Trop(X_0)\ :=\ P_0\left( \Trop\left(\b X_0\cap (\KK^*)^s\right)\right)
\ \subseteq\ \QQ^n.
\end{gather*}
It is possible to establish a fan structure on $\Trop(X_0)$, see, e.g.,~\cite{BoJeSpeStu, MacStu};  we will denote this fan by~$\Upsilon$. 
\end{construction}

Note that in Construction~\ref{con:trop}, despite the infinite intersection, $\Trop\left(\b X_0\cap (\KK^*)^s\right)$ can be computed using finitely many Gr\"obner basis computations, see~\cite{BoJeSpeStu}.

\smallskip
\textit{Step 3: determining the modification.}
From the previous steps, we obtain two polyhedral fans $\Sigma_0, \Upsilon$ in $\QQ^n$. We use the second one to determine a modification of the first one as follows; this is similar to a procedure proposed by Tevelev in a related context, see Remark~\ref{rem:tevelev}.
Firstly, we  compute the fan
$$
\Sigma'\ :=\ \left\{\sigma\cap \tau \:|\: \tau \in \Upsilon\right\}
$$
with support $|\Sigma'| = \Trop(X_0)$. 
Note that, by construction and~\cite[Lem.~2.2]{Tev}, the toric variety $Z'$ corresponding to $\Sigma'$ is still a toric ambient variety of $X_0$. Secondly, we resolve singular cones of $\Sigma'$ by inserting new rays, compare~\cite[11.1]{toricvarieties} for the procedure. This yields a regular fan $\Sigma$ and a toric resolution $Z\to Z'$, where $Z$ is the toric variety corresponding to $\Sigma$.
The following diagram summarizes the situation; 
the map $\pi$ is the composition of the toric morphism defined by the map of fans $\Sigma'\to \Sigma_0$ and the resolution, and $X$ is the proper transform of $X_0$ under~$\pi$ (see also the next step).

\[
\xymatrix{
X 
\ar[d]
\ar@/_2pc/[dd]_\pi
\ar@{}[r]|\subseteq
&
Z
\ar@/^2pc/[dd]^\pi
\ar[d]_{\text{resolve}}
\\
X_0
\ar@{=}[d]
\ar@{}[r]|\subseteq
&
Z'
\ar[d]_{\cap \Upsilon}
\\
X_0
\ar@{}[r]|\subseteq
&
Z_0
}
\]

For the next step, we write the primitive generators $v_1,\ldots,v_m$ 
of the rays added to $\Sigma_0$ in this process
into the columns of the enlarged $(s+m)\times n$-matrix $P := [P_0,v_1,\ldots,v_m]$,
and we compute a Gale dual matrix $Q$ of~$P$.

\smallskip
\textit{Step 4: modifying the ideal of $\Spec(\Cox(X_0))$.}
Finally, we apply~\cite[Alg.~3.6]{HaKeLa}  to compute the Cox ring of the proper transform $X\to X_0$ under~$\pi$;
its main idea is as follows.
Extending the diagram from the previous step 
to the spectra $\overline Y := \spec(\Cox(Y))$ 
of the respective Cox rings, we obtain 
\[
\xymatrix{
\overline X
\ar@/^1pc/[rrr]^{\subseteq}
\ar@{-->}[r]
\ar[d]^{\overline \pi}
&
X 
\ar[d]^\pi
\ar@{}[r]|\subseteq
&
Z
\ar[d]^\pi
&
\KK^{s+m}
\ar[d]^{\overline \pi}
\ar@{-->}[l]
\\
\overline X_0
\ar@{-->}[r]
\ar@/_1pc/[rrr]_{\subseteq}
&
X_0
\ar@{}[r]|\subseteq
&
Z_0
&
\ar@{-->}[l]
\KK^s
}
\]
where a dashed arrow from some $\overline Y$ to the corresponding~$Y$ means that~$Y$
is the quotient of an open subset of $\overline Y$ by a quasitorus as in Construction~\ref{con:bunchedring}.
The spectrum of the Cox ring of the proper transform is the closure
$$
\overline X\ =\ 
\overline{
{\overline \pi}^{-1}(\overline X_0\cap (\KK^*)^s)
}
\ \subseteq\ \KK^{s+m}.
$$
This gives a recipe to compute $\Cox(X)=\OO(\overline X)$:
using the presentation $\Cox(X_0)=\KT{s}/I_0$, compute the pullback $\overline \pi^{*}(I_0)\subseteq \KK[T_1^{\pm 1},\ldots,T_{s+m}^{\pm 1}]$ and then the saturation 
$I:=\overline \pi^{*}(I_0):(T_1\cdots T_{s+m})^\infty$ in the ring $\KK[T_1,\ldots,T_{s+m}]$.
By~\cite[Alg.~3.6]{HaKeLa},  if all variables $T_i$ define prime elements in $R:=\KT{s+m}/I$
and $\dim(I)-\dim(\<T_i,T_j\>+I)\geq 2$ for all $i\ne j$, then $R$ is the Cox ring of~$X$.

It then remains to check whether $X$ is smooth -- if it is, then $X$ is a resolution of~$X_0$, see the proof below. We also provide an algorithm for testing smoothness in the current setting, see Algorithm~\ref{algo:issmooth}.

\medskip

The following algorithm summarizes the process of computing and verifying a candidate for the Cox ring $\Cox(X)$ of a resolution $X\to X_0$. Note that it is close to~\cite[Alg.~2.4.8]{Ke:diss} which in turn is based on~\cite[Thm~3.4.4.9]{ArDeHaLa} where the complete, complexity-one-case is treated. See also~\cite[Ch.~3]{Hu:diss} and check~\cite[Sect.~3]{HaWro} for the case of affine $\KK^*$-surfaces.

\begin{algorithm}[Candidate for the Cox ring of a resolution $X\to X_0=\KK^n/G$]
\label{algo:resolve}
\emph{Input: }
a finite small group $G\subseteq \GL(n)$.
 \begin{itemize}
  \item 
  Compute the Cox ring $\Cox(X_0) = \KT{s}/I_0$ of $X_0$ and the matrix $Q_0$ with Algorithm~\ref{algo:coxquotsing}.
  \item   
  Determine a matrix $P_0$ that is Gale dual to $Q_0$.
  \item 
  Define the fan $\Sigma_0:={\rm faces}(\sigma)$ with the cone $\sigma := P_0(\QQQ^n)\subseteq \QQ^n$.
  \item
  Compute a fan $\Upsilon$ with support  $\Trop(X_0)$
  to obtain $\Sigma':=\{\sigma\cap \tau \:|\:\tau \in \Upsilon\}$.
  \item 
  Stellarly subdivide $\Sigma'$ at primitive vectors $v_1,\ldots,v_m\in \ZZ^n$ until we obtain a regular fan $\Sigma$.
  \item
  Compute a Gale dual $Q$ of the enlarged matrix~$P := [P_0,v_1,\ldots,v_m]$.
  \item 
  Compute a presentation $R:=\KT{r}/I$, where $r:=s+m$, of the Cox ring of the proper transform $X\to X_0$ under~$\pi$ using~\cite[Alg.~3.6]{HaKeLa} with the \texttt{'verify'}-option, as explained above.
  \item
   Verify smoothness of the variety $X$ determined by  $(R,\Sigma)$, e.g. with Algorithm~\ref{algo:issmooth}.
 \end{itemize}
\emph{Output: }
$R$ and $Q$. If the verifications of the last two steps were successful, then $R$ 
is the Cox ring of a resolution $X\to X_0$
and $Q$ is the degree matrix of~$R$.
\end{algorithm}

\begin{lemma}\label{lem:normal}
In the setting of Algorithm~\ref{algo:resolve},
if $\spec(R)\cap (\CC^*)^r$ is smooth and all variables $T_1,\ldots,T_r$ are prime, then $R$ is normal.
\end{lemma}

\begin{proof}
Write 
$\b X_i$ for the localization of
$\b X := \spec(R)$
at the product $T_1\cdots T_i$.
By assumption, $\widehat X := \spec(R)\cap (\CC^*)^r$ is smooth
and therefore also $\overline X_r$ is smooth
and in particular normal.
Since $T_r$ is prime, this implies normality of $\b X_{r-1}$, see~\cite[Lem.~IV.1.2.7]{ArDeHaLa}.
Iterating this argument, $\b X_0 = \spec(R)$ is normal.
\end{proof}

\begin{proof}[Proof of Algorithm~\ref{algo:resolve}]
By construction and~\cite[Alg.~3.6]{HaKeLa}, $X\to X_0$ is a modification and $R$ the Cox ring of~$X$. Note that \cite[Alg.~3.6]{HaKeLa} directly translates to the affine setting and the required normality test of $R$ is provided by Lemma~\ref{lem:normal} if the smoothness requirement is ensured (e.g. by explicit computations). 
\end{proof}

\begin{remark}
 If the smoothness test in Algorithm~\ref{algo:resolve} is left out,  $R$ is still the Cox ring of a modification of~$X_0$
 provided that $R$ is normal.
\end{remark}

 We can test smoothness of $X$ using the following straight-forward test; compare also~\cite{HaKe}.
The idea is to first test the toric charts $Z_\sigma$, where $\sigma$ is a cone of the fan of the canonical toric ambient variety $Z$, for being smooth; it then suffices to check the smoothness of $X_\sigma := Z_\sigma\cap X$ if all $Z_\sigma$ are smooth.
\[
 \xymatrix{
 &
 \overline X
 \ar@{-->}[d]
 \ar@{}[r]|\subseteq
 &
 \KK^r
 \ar@{-->}[d]
 &
 \\
 X_\sigma
  \ar@{}[r]|\supseteq
  \ar@/_1.2pc/[rrr]_\subseteq
 &
 X
 \ar[r]
 &
 Z
  \ar@{}[r]|\subseteq
 &
 Z_\sigma
 }
\]
For the latter, we use the following ad-hoc algorithm which directly looks for singularities of $X$ that lie in $Z_\sigma \cong (\KK^*)^l\times \KK^k$ for suitable $k,l\in\ZZ_{\geq 0}$.

\begin{algorithm}[Smoothness test]
\label{algo:issmooth}
\emph{Input: } a normal, irreducible $A_2$-maximal variety~$X$ with $\Cox(X)$ finitely generated.
\begin{itemize}
 \item 
 Compute the fan $\Sigma$ of the canonical toric ambient variety $Z$ of~$X$.
 \item 
 Return \texttt{false} if $\Sigma$ is not regular.
 \item
 Return \texttt{true} if and only if all $X_\sigma$ where $\sigma \in \Sigma$ are smooth.
\end{itemize}
{\em Output: } \texttt{true} if $X$ is smooth, \texttt{false} otherwise.
\end{algorithm}

\begin{remark}
In Algorithm~\ref{algo:issmooth},
the smoothness of $X_\sigma$ can be checked as follows.
According to~\cite[Constr.~3.3.1.1]{ArDeHaLa},
we can compute a  pairwise disjoint decomposition 
\begin{gather*}
X\,=\,X_1\sqcup\ldots \sqcup X_u
\end{gather*}
For each $i$, one then produces equations 
for $X_i$ in $\OO(\KK^m\times (\KK^*)^{n-m})\cong \KK[\sigma^\dual \cap \ZZ^n]$.
 The test then boils down to $X_i^{\rm sing}\cap \KK^m\times \TT^{n-m}$ being empty or not. This can be done by a saturation computation using Gr\"obner bases.
\end{remark}

\begin{remark}
Using the torus action on $X$ and the fact that the singular locus is closed one may often reduce the number of sets $X_{\sigma}$ for which the computations have to be performed to check the smoothness.
\end{remark}

\begin{example}
\label{ex:Q8v2}
Consider the Cox ring $\Cox(X_0)=\KK[T_1,T_2,T_3]/I_0$ 
computed in Example~\ref{ex:Q8v1}
with its degree-map $Q\colon \ZZ^3\to (\ZZ/2\ZZ)^2$.
We apply the steps of Algorithm~\ref{algo:resolve} to $X_0$.
The canonical toric ambient variety is the affine toric variety $Z_0=Z_0(\sigma)$ with $\sigma\subseteq \QQ^3$ being the polyhedral cone spanned by $(1, 0, 0), (1, 2, 0), (1, 0, 2)$.
The following picture shows the steps for the toric resolution.
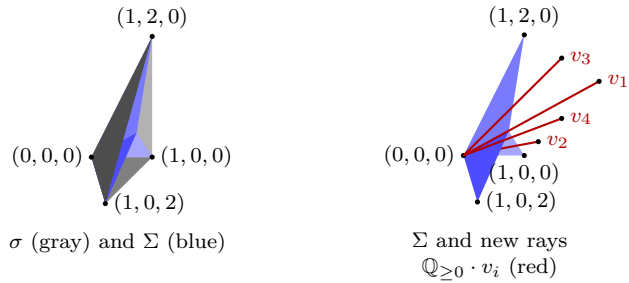
\begin{center}
\footnotesize
\begin{minipage}{3cm}
\begin{center}
\begin{tikzpicture}[scale=.8]
\fill[color=black!30] (0,0,0) -- (1,0,0) -- (1,2,0);
\fill[color=black!50] (0,0,0) -- (1,0,0) -- (1,0,2);

\fill[color=blue!35] (0,0,0) -- (1,0,0) -- (1,.66,.66);
\fill[color=blue!53] (0,0,0) -- (1,2,0) -- (1,.66,.66);
\fill[color=blue!70] (0,0,0) -- (1,0,2) -- (1,.66,.66);

\fill[color=black!70] (0,0,0) -- (1,0,2) -- (1,2,0);

\fill (0,0,0) circle(1.2pt) node[anchor=east]{$(0,0,0)$};
\fill (1,0,0) circle(1.2pt) node[anchor=west]{$(1,0,0)$};
\fill (1,0,2) circle(1.2pt) node[anchor=west]{$(1,0,2)$};
\fill (1,2,0) circle(1.2pt) node[anchor=south]{$(1,2,0)$};
\end{tikzpicture}
\\
$\sigma$ (gray) and $\Sigma$ (blue) \phantom{and new rays $\QQ$}
\end{center}
\end{minipage}
\qquad\quad 
\qquad\quad 
\begin{minipage}{3cm}
\begin{center}
\begin{tikzpicture}[scale=.8]
\fill[color=blue!35] (0,0,0) -- (1,0,0) -- (1,.66,.66);
\fill[color=blue!52] (0,0,0) -- (1,2,0) -- (1,.66,.66);

\draw[color=red!70!black,thick] (0,0,0) -- (2,1,2) node[anchor=west]{$v_2$};
\fill[color=blue!70] (0,0,0) -- (1,0,2) -- (1,.66,.66);

\draw[color=red!70!black,thick] (0,0,0) -- (3,2,2) node[anchor=west]{$v_1$};
\fill (3,2,2) circle(1.2pt);

\fill (0,0,0) circle(1.2pt) node[anchor=east]{$(0,0,0)$};
\fill (1,0,0) circle(1.2pt) node[anchor=north]{$(1,0,0)$};
\fill (1,0,2) circle(1.2pt) node[anchor=west]{$(1,0,2)$};
\fill (1,2,0) circle(1.2pt) node[anchor=south]{$(1,2,0)$};

\draw[color=red!70!black,thick] (0,0,0) -- (2,2,1) node[anchor=west]{$v_3$};
\draw[color=red!70!black,thick] (0,0,0) -- (2,1,1) node[anchor=west]{$v_4$};

\fill (2,2,1) circle(1.2pt);
\fill (2,1,2) circle(1.2pt);
\fill (2,1,1) circle(1.2pt);

\end{tikzpicture}
\\
$\Sigma$ and new rays $\QQQ\cdot v_i$ (red)
\end{center}
\end{minipage}
\end{center}
The two-dimensional fan $\Sigma = \{\sigma\}\cap\Trop(X_0)$ has already the additional ray $\QQQ\cdot v_1$ where $v_1 := (3,2,2)$. The fan can be resolved by insertion of the further rays through
$$
v_2\,:=\,(2, 1, 2),\qquad
v_3\,:=\,(2, 2, 1),\qquad
v_4\,:=\,(2, 1, 1).
$$
This yields a toric resolution $Z\to Z_0$
which then induces a resolution $X\to X_0$ as the proper transform of $Z\to Z_0$;
its $\ZZ^4$-graded Cox ring $\Cox(X)$ and degree matrix are
\begin{gather*}
\Cox(X)\ =\ \KT{7}/\<
{4}T_{1}^{2}T_{7}-T_{6}T_{2}^{2}+T_{3}^{2}T_{5}
\>,\\
\left[
\mbox{\tiny $
\begin{array}{rrrrrrr}
-1 & -1 & -1 & 1 & 0 & 0 & 0\\
0 & -1 & 0 & 1 & -1 & 1 & -1\\
0 & 0 & -1 & 1 & 1 & -1 & -1\\
0 & -1 & -1 & 0 & 1 & 1 & -1
\end{array}
$}
\right].
\end{gather*}
One verifies that $X$ is smooth. As $X$ is a surface, we can verify minimality of the  resolution by inspecting the self intersection numbers of divisors corresponding to $V(\widehat X;\,T_i)$: they  are 
$-1, -1, -1, -2, -2, -2, -2$.
\qed
\end{example}

We will use Algorithms~\ref{algo:coxquotsing} and~\ref{algo:resolve} in Section~\ref{sec:quotsing}
to compute resolutions of quotient singularities. The next section explains an implementation of the algorithms of this section.
We close this section with the following remark.

\begin{remark}
\label{rem:tevelev}
 A result by Tevelev~\cite{Tev} indicates that the steps 
 performed in Algorithm~\ref{algo:resolve}
should work,
provided that we find the correct embedding $X\to Z$ 
and fan structure on $\Trop(X)$.
Moreover, in~\cite{Tev2},
Tevelev shows that 
any embedded resolution of singularities 
is induced by an equivariant map of toric varieties
in the following sense.
Consider a birational morphism $\pi\colon Y_2\to Y_1$
of smooth projective varieties with exceptional locus $D\subseteq Y_2$ such that the restriction of $\pi$ to $Y_2 \supseteq X_2 \to X_1 \subseteq Y_1$ is a resolution of singularities. Assume that $D$ has simple normal crossings and intersects $X_2$ transversally. Then there are embeddings $Y_1\subseteq \PP^n$
and $Y_2\subseteq Z_2$ with a smooth toric variety $Z_2$
such that $X_2$ and $Y_2$ intersect the toric boundary of $Z_2$
transversally and $\pi$ is a restriction of a toric morphism $\overline{\pi} \colon Z_2 \to Z_1$.
\[
\xymatrix{
X_2
\ar[d]^{\pi}
\ar[r]
&
Y_2 
\ar[d]^\pi
\ar[r]
&
Z_2
\ar[d]^{\overline{\pi}}
\\
X_1
\ar[r]
&
Y_1
\ar[r]
&
Z_1 = \PP^n
}
\]
Hence the main task seems to be to find the correct embeddings $X_i\subseteq Z_i$ -- it would be interesting to investigate this in further work.
\end{remark}

\section{Implementation}
\label{sec:implementation}
We have implemented Algorithms~\ref{algo:coxquotsing} and~\ref{algo:resolve} in the library \texttt{quotsingcox.lib} for the Open Source computer algebra system \texttt{Singular}~\cite{singular}. 
It will be made available at~\cite{quotsingcox}.
We currently make use of the interface to \texttt{Normaliz}~\cite{Normaliz}.
We shortly explain its use in the following two examples.

\begin{example}[Algorithm~\ref{algo:coxquotsing}]
\label{ex:implem1}
 We recompute Example~\ref{ex:Q8v1} using the library  \texttt{quotsingcox.lib}~\cite{quotsingcox}. 
 In~\texttt{Singular}, we first load the built-in library for convex geometry (implemented by Y.~Ren) and our library with
\\[1ex]
\begingroup
\noindent
\footnotesize
\texttt{> LIB \char`\" gfanlib.so\char`\";}
\texttt{ LIB \char`\" quotsingcox.lib\char`\";}\\[1ex]
\endgroup
\noindent
Next, we enter generators for $G\subseteq \GL(2)$. We first have to define a ring where not only the generators of $G$ are defined but also all eigenvalues of~$[G,G]$; $\QQ(i)$ will work here:\\[1ex]
\begingroup
\noindent
\footnotesize
\texttt{> ring R = (0,CplxUnit),T(1..2),dp; minpoly = CplxUnit\textasciicircum 2 + 1;}\\
\texttt{>\ matrix M0[2][2] =\ CplxUnit,0, 0,-CplxUnit;}\\
\texttt{>\ matrix M1[2][2] =  0,-CplxUnit,-CplxUnit,0;}\\
\texttt{>\ list G = M0, M1; // the group G}\\[1ex]
\endgroup
\noindent
Next, we have to enter generators for the derived subgroup. For integer entries, you can also use the command \texttt{list GG = gapDerivedSubgroup(G);} which writes a file and applies \texttt{GAP}~\cite{GAP4}.\\[1ex]
\begingroup
\noindent
\footnotesize
\texttt{>\ matrix M2[2][2] = -1,0,0,-1;}\\
\texttt{>\ list GG = M2; // derived subgroup of G}\\[1ex]
\endgroup
\noindent
We can then compute the Cox ring $\Cox(X_0)$ and investigate it using the following commands:\\[1ex]
\begingroup
\noindent
\footnotesize
\texttt{>\ def R0 = coxquot(G, GG); setring R0; R0;}\\
\texttt{//   characteristic : 0}\\
\texttt{//   1 parameter    : CplxUnit }\\
\texttt{//   minpoly        : (CplxUnit\textasciicircum 2+1)}\\
\texttt{//   number of vars : 3}\\
\texttt{//        block   1 : ordering dp}\\
\texttt{//                  : names    T(1) T(2) T(3)}\\
\texttt{//        block   2 : ordering C}\\
\texttt{>\ ideal I0 = Inew; I0; // Cox ring of X0 is R0/I0}\\
\texttt{I0[1]=T(1)\textasciicircum 2-T(2)\textasciicircum 2-4*T(3)\textasciicircum 2}\\[1ex]
\endgroup
\noindent
Generators for $\Cl(X_0)$ as a matrix subgroup \texttt{Gnew} of $\GL(3)$ and the degrees $\deg(T_i)\in \Cl(X_0)$ of the generators of $\Cox(X_0)$ can be displayed 
with \texttt{print(Gnew[1])} and 
 \texttt{print(Gnew[2])}. This yields the matrices with columns
 $-e_1,-e_2,e_3$ and $-e_2,-e_1,-e_3$, respectively.

\begingroup
\noindent
\footnotesize
\texttt{>\ Wnew; // degrees of the generators (eigenvalues)}\\[1ex]
\texttt{[1]:}\\
\texttt{\ \ \_[1,1]=-1}\\
\texttt{\ \ \_[2,1]=-1}\\
\texttt{[2]:}\\
\texttt{\ \ \_[1,1]=-1}\\
\texttt{\ \ \_[2,1]=1}\\
\texttt{[3]:}\\
\texttt{\ \ \_[1,1]=1}\\
\texttt{\ \ \_[2,1]=-1}
\endgroup
\noindent
\qed
\end{example}

\begin{example}[Algorithm~\ref{algo:resolve}]
Building on Example~\ref{ex:implem1}, we recompute Example~\ref{ex:Q8v2}.
The next step has to be prepared manually: one has to give the isomorphism type of \texttt{Gnew} as a product $\prod \ZZ/n_i\ZZ$.
Here, \texttt{gapStructureDescription(Gnew);} can be used to see that $n_1=n_2=2$.
We store the integers $n_i$, translate the degrees given in \texttt{Wnew} into the columns of a matrix $Q_0$ and compute a Gale dual matrix $P_0$ for~$Q_0$:\\[1ex]
\begingroup
\noindent
\footnotesize
\texttt{>\ list ClX0 = 2, 2; // ZZ/2ZZ x ZZ/2ZZ}\\
\texttt{>\ intmat Q0[2][3] = 1,1,0,1,0,1; // degree matrix}\\
\texttt{>\ intmat P0 = finiteGaleDual(Q0, ClX0); P0;}\\
\texttt{\ \ 1,1,1}\\
\texttt{\ \ 1,2,0}\\
\texttt{\ \ 0,0,2}\\[1ex]
\endgroup
\noindent
We can then
compute, verify and print a candidate for the Cox ring
of a resolution $X\to X_0$ using Algorithm~\ref{algo:resolve} with the following commands.
We present the Cox ring as $\KT{s}/I_0$; the algorithm returns the ideal~$I_0$. \\[1ex]
\begingroup
\noindent
\footnotesize
\texttt{>\ def R = resolveQuotSing(I0, P0, 1); setring R; R; // 1 means verify}\\
\texttt{1st part of Verification successful: please check primality of the variables (e.g. with primeVars).}\\
\texttt{//   characteristic : 0}\\
\texttt{//   1 parameter    : CplxUnit }\\
\texttt{//   minpoly        : (CplxUnit\textasciicircum 2+1)}\\
\texttt{//   number of vars : 7}\\
\texttt{//        block   1 : ordering dp}\\
\texttt{//                  : names    T(1) T(2) T(3) T(4) T(5) T(6) T(7)}\\
\texttt{//        block   2 : ordering C}\\
\texttt{>\ Inew; // Cox ring of X is R/Inew}\\
\texttt{Inew[1]=T(1)\textasciicircum 2*T(5)-T(2)\textasciicircum 2*T(6)-4*T(3)\textasciicircum 2*T(7)}\\[1ex]
\endgroup
\noindent
One can check the primality of the variables $T_i$ using the command \texttt{primeVars(Inew);}.
The $\ZZ^4$-grading is stored in the new degree matrix $Q$; it is obtained as a Gale dual matrix $Q$ to the new matrix~$P$:\\[1ex]
\begingroup
\noindent
\footnotesize
\texttt{>\ intmat P = L[2]; }\\
\texttt{>\ intmat Q = gale(P); Q; }\\
\texttt{\ \ -1,-1,-1,1,0,0,0,}\\
\texttt{\ \ -2,-1,-1,0,2,0,0,}\\
\texttt{\ \ 3,0,1,0,-4,2,0,}\\
\texttt{\ \ 2,0,0,0,-3,1,1 }
\endgroup
\noindent
\qed
\end{example}

\section{Resolutions of $3$-dimensional quotient singularities}
\label{sec:quotsing}

\subsection{Cox rings}
In this section, we compute Cox rings of resolutions of quotient singularities $\KK^3/G$ where~$G$ is a group of order at most $12$ and discuss the results.
As explained in Remark~\ref{rem:pseudo-reflections}, we will consider only faithful representations of~$G$ without pseudo-reflections.

\begin{notation}
In the following, we denote by $S_3$ the permutation group of three elements,
by $D_{2n}$ the dihedral group of $2n$ elements, by $Q_8$ the quaternion group
and by $A_4$ the alternating group of four elements.
Moreover, we write $\BD_3$ for the {\em binary dihedral group}, i.e., the abstract group with $12$ elements 
$$
\BD_3
\ =\ 
\left\langle a,x;\,a^{6} = 1,\,x^2 = a^3,\,x^{-1}ax = a^{-1}
\right\rangle
.
$$
\end{notation}

\begin{proposition}
\label{prop:reps}
 Up to conjugacy in $\GL(3)$, the $3$-dimensional faithful small representations 
  of non-abelian groups up to order $12$ are as follows.
 
\begingroup
\footnotesize
\def\arraystretch{0.6}
\begin{center}
\begin{minipage}{5cm}
\setlength{\arraycolsep}{1pt}
\begin{longtable}{ccl}
\myhline
No. & $G$ & gen.s in $\GL(3)$\\
\myhline\\
$1$ & $S_3$ & 
$
\left[\mbox{\tiny $
\begin{array}{rrr}
  0 & -1 & 0 \\
 1 & -1 & 0 \\ 
 0 & 0 & 1
 \end{array}
 $}\right],\quad 
 \left[\mbox{\tiny $
\begin{array}{rrr}
-1 & 1 & 0 \\
0 & 1 & 0 \\ 
0 & 0 &-1
 \end{array}
 $}\right].
$
\\\\\hline\\
$2$ & $D_8$ &  $\left[\mbox{\tiny $
\begin{array}{rrr}
0&-1&0\\
1&0&0\\
0&0&1\end{array}
 $}\right],\quad \left[\mbox{\tiny $
\begin{array}{rrr}
1&0&0\\
0&-1&0\\
0&0&-1
\end{array}
 $}\right].$

\\\\\hline\\
$3$ & $Q_8$ & $\left[\mbox{\tiny $
\begin{array}{rrr}
i & 0 & 0 \\
0 & -i & 0 \\
0 & 0 & 1\end{array}
 $}\right],
\quad \left[\mbox{\tiny $
\begin{array}{rrr}
0 & -1 & 0 \\
1 & 0 & 0 \\
0 & 0 & 1
\end{array}
 $}\right]
$. 
\\\\\hline\\
$4$ & $Q_8$ & $\left[\mbox{\tiny $
\begin{array}{rrr}
i & 0 & 0 \\
0 & -i & 0 \\
0 & 0 & -1\end{array}
 $}\right],
\quad \left[\mbox{\tiny $
\begin{array}{rrr}
0 & -1 & 0 \\
1 & 0 & 0 \\
0 & 0 & 1
\end{array}
$}\right]
$. 
\\\\\hline\\
$5$ & $D_{10}$ & $\left[\mbox{\tiny $
\begin{array}{rrr}
\zeta_5 & 0 & 0 \\
0 & \zeta_5^4 & 0 \\
0 & 0 & 1\end{array}
 $}\right],
\quad \left[\mbox{\tiny $
\begin{array}{rrr}
0 & 1 & 0 \\
1 & 0 & 0 \\
0 & 0 & -1
\end{array}
$}\right]$. 
\\\\
\myhline
\end{longtable}

\end{minipage}
\qquad 
\begin{minipage}{5cm}
\setlength{\arraycolsep}{1pt}
\begin{longtable}{ccl}
\myhline
No. & $G$ & gen.s in $\GL(3)$\\
\myhline\\
$6$ & $D_{12}$ & $\left[\mbox{\tiny $
\begin{array}{rrr}
-\zeta_3^2 & 0 & 0 \\
0 & -\zeta_3 & 0 \\
0 & 0 & 1\end{array}
 $}\right],
\quad \left[\mbox{\tiny $
\begin{array}{rrr}
0 & 1 & 0 \\
1 & 0 & 0 \\
0 & 0 & -1
\end{array}
 $}\right].$  
\\\\\hline\\
$7$ & $A_{4}$ & $\left[\mbox{\tiny $
\begin{array}{rrr}
0 & 0 & 1 \\
1 & 0 & 0 \\
0 & 1 & 0\end{array}
 $}\right],
\quad \left[\mbox{\tiny $
\begin{array}{rrr}
-1 & 0 & 0 \\
0 & -1 & 0 \\
0 & 0 & 1
\end{array}
 $}\right].
$ 
\\\\\hline\\
$8$ & $\BD_3$ & $\left[\mbox{\tiny $
\begin{array}{rrr}
-\zeta_3^2 & 0 & 0 \\
0 & -\zeta_3 & 0 \\
0 & 0 & 1\end{array}
 $}\right],
\quad \left[\mbox{\tiny $
\begin{array}{rrr}
0 & -i & 0 \\
-i & 0 & 0 \\
0 & 0 & 1
\end{array}
$}\right].
$ 
\\\\\hline\\
$9$ & $\BD_3$ & $\left[\mbox{\tiny $
\begin{array}{rrr}
-\zeta_3^2 & 0 & 0 \\
0 & -\zeta_3 & 0 \\
0 & 0 & -1\end{array}
 $}\right],
\quad \left[\mbox{\tiny $
\begin{array}{rrr}
0 & -i & 0 \\
-i & 0 & 0 \\
0 & 0 & i
\end{array}
$}\right]
$.
\\\\\hline\\
$10$ & $\BD_3$ & $\left[\mbox{\tiny $
\begin{array}{rrr}
-\zeta_3^2 & 0 & 0 \\
0 & -\zeta_3 & 0 \\
0 & 0 & 1\end{array}
 $}\right],
\quad \left[\mbox{\tiny $
\begin{array}{rrr}
0 & -i & 0 \\
-i & 0 & 0 \\
0 & 0 & -1
\end{array}
$}\right]
$. 
\\\\
\myhline
\end{longtable}

\end{minipage}
\end{center}
\endgroup
In the table, $i\in \CC$ denotes the imaginary unit,
and $\zeta_k\in \CC$ is a primitive $k$-th root of unity.
All listed representations are reducible except for the $A_4$-case.
\end{proposition}

\begin{proof}
To classify these representations we use the library of groups of small order, which is a part of \texttt{GAP}~\cite{GAP4}. We do the following steps:
\begin{enumerate}
\item view the character tables to get dimensions of irreducible representations,
\item find all irreducible representations: for a group of small order and a given character one can either locate a representation in the literature or construct easily,
\item combine irreducible representations to get all faithful 3-dimensional representations of a given group (in particular, direct sums of 1-dimensional representations are not allowed since they are not faithful representations -- the image is an abelian group),
\item eliminate representations with pseudo-reflections (see Remark~\ref{rem:pseudo-reflections}),
\item check whether obtained representations have different groups as images (if representations differ just by a permutation of conjugacy classes in $G$, they give the same quotient); for any two representations of the same group on the list above the set of matrix traces are different, hence this condition is satisfied.\qedhere
\end{enumerate}
\end{proof}

\begin{proposition}
\label{prop:coxX0}
In the setting of Proposition~\ref{prop:reps},
the Cox rings of $X_0:=\KK^3/G$
are as listed the following table.

\begingroup
\footnotesize
\def\arraystretch{0.5}
\begin{longtable}{ccccc}
\myhline
No.  & $G$ & $Cl(X_0)$ & degree matrix & $\Cox(X_0)$\\
\myhline\\
$1$ & $S_3$ & $\ZZ/2\ZZ$ & $\left[\mbox{\tiny $
\begin{array}{rrrr}
\b 1 & \b 1 & \b 0 & \b 0
\end{array}
$}\right]$  
& 
$\begin{array}{c}
 \KT{4}/I \text{ with $I$ gen.~by }
 \\
 4T_{3}^3-T_{2}^2-27T_{4}^2
\end{array}
$
\\\\\hline\\
$2$ & $D_8$ & $\ZZ/2\ZZ\times \ZZ/2\ZZ$ & $\left[\mbox{\tiny $
\begin{array}{rrrr}
\b 1 & \b  1 &  \b 0 & \b  0 \\
\b 1 &  \b 0 & \b  1 & \b  0
\end{array}
$}\right]$  
& 
$\begin{array}{c}
 \KT{4}/I \text{ with $I$ gen.~by }
 \\
4T_{1}^2+T_{2}^2-T_{4}^2
\end{array}
$
\\\\\hline\\
$3$ & $Q_8$ & $\ZZ/2\ZZ\times \ZZ/2\ZZ$ & $\left[\mbox{\tiny $
\begin{array}{rrrr}
\b 1 & \b  1 & \b  0 & \b  0 \\
\b 1 & \b  0 & \b  1 & \b  0
\end{array}
$}\right]$  
& 
$\begin{array}{c}
 \KT{4}/I \text{ with $I$ gen.~by }
 \\
T_{1}^2-T_{2}^2+4T_{3}^2
\end{array}
$
\\\\\hline\\
$4$ & $Q_8$ & $\ZZ/2\ZZ\times \ZZ/2\ZZ$ & $\left[\mbox{\tiny $
\begin{array}{rrrr}
\b   1 & \b  1 & \b  1 & \b  0 \\
 \b  1 & \b  0 & \b  0 & \b  1
\end{array}
$}\right]$  
& 
$\begin{array}{c}
 \KT{4}/I \text{ with $I$ gen.~by }
 \\
T_{1}^2-T_{3}^2+4T_{4}^2
\end{array}
$
\\\\\hline\\
$5$ & $D_{10}$ & $\ZZ/2\ZZ$ & $\left[\mbox{\tiny $
\begin{array}{rrrr}
 \b 1 & \b  1 & \b  0 & \b  0
\end{array}
$}\right]$  
& 
$\begin{array}{c}
 \KT{4}/I \text{ with $I$ gen.~by }
 \\
4T_{3}^5+T_{2}^2-T_{4}^2
\end{array}
$
\\\\\hline\\
$6$ & $D_{12}$ & $\ZZ/2\ZZ\times \ZZ/2\ZZ$ & $\left[\mbox{\tiny $
\begin{array}{rrrr}
 \b  1 & \b  1 & \b  0 & \b  0 \\
 \b  1 & \b  0 & \b  1 & \b  0
\end{array}
$}\right]$  
& 
$\begin{array}{c}
 \KT{4}/I \text{ with $I$ gen.~by }
 \\
4T_{4}^3+T_{1}^2-T_{2}^2
\end{array}
$
\\\\\hline\\
$7$ & $A_4$ & $\ZZ/3\ZZ$ & $\left[\mbox{\tiny $
\begin{array}{rrrr}
 \b  0 & \b  0 & \b  2 & \b  1
\end{array}
$}\right]$  
& 
$\begin{array}{c}
 \KT{4}/I \text{ with $I$ gen.~by }
 \\
T_{1}^3+T_{3}^3-3T_{1}T_{3}T_{4}+T_{4}^3-27T_{2}^2
\end{array}
$
\\\\\hline\\
$8$ & $BD_3$ & $\ZZ/4\ZZ$ & $\left[\mbox{\tiny $
\begin{array}{rrrr}
 \b  3 & \b  1 & \b  2 & \b  0
\end{array}
$}\right]$  
& 
$\begin{array}{c}
 \KT{4}/I \text{ with $I$ gen.~by }
 \\
4T_{3}^3+T_{1}^2-T_{2}^2
\end{array}
$
\\\\\hline\\
$9$ & $BD_3$ & $\ZZ/4\ZZ$ & $\left[\mbox{\tiny $
\begin{array}{rrrr}
 \b  3 &\b  1 & \b 1 & \b 2
\end{array}
$}\right]$  
& 
$\begin{array}{c}
 \KT{4}/I \text{ with $I$ gen.~by }
 \\
4T_{4}^3+T_{1}^2-T_{3}^2 
\end{array}
$
\\\\\hline\\
$10$ & $BD_3$ & $\ZZ/4\ZZ$ & $\left[\mbox{\tiny $
\begin{array}{rrrr}
  \b  3 & \b 1 & \b 2 & \b 2
\end{array}
$}\right]$  
& 
$\begin{array}{c}
 \KT{4}/I \text{ with $I$ gen.~by }
 \\
4T_{4}^3+T_{1}^2-T_{2}^2
\end{array}
$
\\\\
\myhline
\end{longtable}
\endgroup
 
\end{proposition}

\begin{proof}
 The listed Cox rings are obtained by applying Algorithm~\ref{algo:coxquotsing} 
 to the list of representations from Proposition~\ref{prop:reps}
 using our implementation~\ref{sec:implementation}.
\end{proof}

\begin{remark}
Note that several rings on the list (without grading considered) are isomorphic. This is because the ring structure of the Cox ring of $\KK^n/G$ is just the invariant ring $\KK[x_1,\ldots,x_n]^{[G,G]}$ and derived subgroups in cases 2, 3, 4 ($D_8$ and $Q_8$ representations) and also 6, 8, 9, 10 ($D_{12}$ and $BD_3$ representations) are conjugate subgroups of~$\GL(3)$.
\end{remark}

\begin{theorem}
\label{thm:dim3}
In the setting of Proposition~\ref{prop:coxX0}, a resolution $X\to X_0$ of the quotient $X_0 = \KK^3/G$ has the following Cox ring, respectively.

\begingroup
\setlength{\tabcolsep}{2pt}
\arraycolsep=1.2pt
\def\arraystretch{0.6}
\footnotesize
\begin{longtable}{ccccc}
\myhline
No.  & $G$ & $Cl(X)$ & degree matrix & $\Cox(X)$\\
\myhline
\\
$1$ & $S_3$ & $\ZZ^4$ & $\left[\mbox{\tiny $
\begin{array}{rrrrrrrr}
0&-3&-2&-3&1&0&0&0\\
-1&-1&0&0&0&2&0&0\\
1&0&-1&-1&0&-2&1&0\\
1&0&-1&-2&0&-3&0&1
\end{array}
$}\right]$  
& 
$\begin{array}{c}
 \KT{8}/I \text{ with $I$ gen.~by }
 \\
4T_{3}^3T_{7}-T_{2}^2T_{6}-27T_{4}^2T_{8}
\end{array}
$
\\
\\\hline \\
$2$ & $D_8$ & $\ZZ^{4}$ & $\left[\mbox{\tiny $
\begin{array}{rrrrrrrr}
-1&-1&0&-1&1&0&0&0\\
-1&0&-1&0&0&2&0&0\\
-1&-1&0&-2&0&0&2&0\\
1&0&0&1&0&-1&-1&1
\end{array}
$}\right]$  
& 
$\begin{array}{c}
 \KT{8}/I \text{ with $I$ gen.~by }
 \\
 -4T_{1}^2T_{6}+T_{4}^2T_{7}-T_{2}^2T_{8}
\end{array}
$
\\\\\hline \\
$3$ & $Q_8$ & $\ZZ^{4}$ & $\left[\mbox{\tiny $
\begin{array}{rrrrrrrr}
-1&-1&-1&0&1&0&0&0\\
-2&-1&-1&0&0&2&0&0\\
1&0&-1&0&0&-2&2&0\\
2&0&0&0&0&-3&1&1
\end{array}
$}\right]$  
& 
$\begin{array}{c}
 \KT{8}/I \text{ with $I$ gen.~by }
 \\
 T_{1}^2T_{6}+4T_{3}^2T_{7}-T_{2}^2T_{8}
\end{array}
$
\\\\\hline \\
$4$ & $Q_8$ & $\ZZ^{4}$ & $\left[\mbox{\tiny $
\begin{array}{rrrrrrrr}
-1&0&-1&-1&1&0&0&0\\
-2&-1&-1&-1&0&2&0&0\\
1&0&0&-1&0&-2&2&0\\
2&1&0&0&0&-3&1&1
\end{array}
$}\right]$  
& 
$\begin{array}{c}
 \KT{8}/I \text{ with $I$ gen.~by }
 \\
 -T_{1}^2T_{6}-4T_{4}^2T_{7}+T_{3}^2T_{8}
\end{array}
$
\\\\\hline \\
$5$ & $D_{10}$ & $\ZZ^{5}$ & $\left[\mbox{\tiny $
\begin{array}{rrrrrrrrr}
0 & -5 & -2 & -5 & 1 & 0 & 0 & 0 & 0 \\
-1 & -1 & 0 & 0 & 0 & 2 & 0 & 0 & 0 \\
1 & 0 & -1 & -1 & 0 & -2 & 1 & 0 & 0 \\
2 & 0 & -1 & -2 & 0 & -4 & 0 & 1 & 0 \\
2 & 0 & -1 & -3 & 0 & -5 & 0 & 0 & 1
\end{array}
$}\right]$  
& 
$\begin{array}{c}
 \KT{9}/I \text{ with $I$ gen.~by }
 \\
 4T_{3}^5T_{7}^3T_{8}+T_{2}^2T_{6}-T_{4}^2T_{9}
\end{array}
$
\\\\\hline \\
$6$ & $D_{12}$ & $\ZZ^{5}$ & $\left[\mbox{\tiny $
\begin{array}{rrrrrrrrr}
-3 & -3 & 0 & -2 & 2 & 0 & 0 & 0 & 0 \\
-1 & 0 & -1 & 0 & 0 & 2 & 0 & 0 & 0 \\
1 & 1 & 0 & 0 & -1 & 0 & 1 & 0 & 0 \\
2 & 1 & 0 & 1 & -1 & -1 & 0 & 1 & 0 \\
2 & 2 & 0 & 1 & -2 & 0 & 0 & 0 & 1
\end{array}
$}\right]$  
& 
$\begin{array}{c}
 \KT{9}/I \text{ with $I$ gen.~by }
 \\
 4T_{4}^3T_{7}^2T_{9}+T_{1}^2T_{6}-T_{2}^2T_{8}
\end{array}
$
\\\\\hline \\
$7$ & $A_4$ & $\ZZ^{3}$ & $\left[\mbox{\tiny $
\begin{array}{rrrrrrrrrr}
 0 & 0 & -1 & -1 &  0 & -1 & -1 &  1 & 1 & 0\\
 0 & 0 &  0 & -1 &  0 &  0 & -1 & -1 & 2 & 0\\
-1 & 0 &  0 &  0 & -1 & -1 & -1 &  0 & 0 & 1\\
\end{array}
$}\right]$  
& 
$\begin{array}{c}
 \KT{10}/I \text{ with $I$ gen.~by }
 \\
T_{4}T_{7}T_{9}+T_{2}T_{6}+T_{3}T_{5},\\
T_{4}^2T_{9}-T_{2}T_{3}-T_{6}T_{10},\\
T_{3}T_{6}T_{8}+T_{2}T_{7}+T_{4}T_{5},\\
T_{3}^2T_{8}-T_{2}T_{4}-T_{7}T_{10},\\
27T_{1}^2T_{4}-T_{6}^2T_{8}+T_{5}T_{7},\\
27T_{1}^2T_{3}-T_{7}^2T_{9}+T_{5}T_{6},\\
T_{6}T_{7}T_{8}T_{9}+27T_{1}^2T_2-T_{5}^2,\\
T_{3}T_{4}T_{8}T_{9}-T_{2}^2+T_{5}T_{10},\\
T_{3}T_{7}T_{8}T_{9}+T_{4}T_{6}T_{8}T_{9}-27T_{1}^2T_{10}+T_{2}T_{5}
\end{array}
$
\\\\\hline \\
$8$ & $BD_3$ & $\ZZ^{5}$ & $\left[\mbox{\tiny $
\begin{array}{rrrrrrrrr}
-3 & -3 & -2 & 0 & 2 & 0 & 0 & 0 & 0 \\
2 & 3 & 2 & 0 & -3 & 2 & 0 & 0 & 0 \\
-1 & -2 & -2 & 0 & 2 & -2 & 1 & 0 & 0 \\
-2 & -5 & -3 & 0 & 5 & -5 & 0 & 1 & 0 \\
-2 & -4 & -3 & 0 & 4 & -4 & 0 & 0 & 1
\end{array}
$}\right]$  
& 
$\begin{array}{c}
 \KT{9}/I \text{ with $I$ gen.~by }
 \\
 4T_{3}^3T_{7}^2T_{9}+T_{1}^2T_{6}-T_{2}^2T_{8}
\end{array}
$
\\\\\hline \\
$9$ & $BD_3$ & $\ZZ^{8}$ & $\left[\mbox{\tiny $
\begin{array}{rrrrrrrrrrrr}
-3 & 0 & -3 & -2 & 1 & 0 & 0 & 0 & 0 & 0 & 0 & 0\\
-1 & -1 & -1 & -2 & 0 & 2 & 0 & 0 & 0 & 0 & 0 & 0\\ 
-1 & 0 & 0 & 2 & 0 & -3 & 2 & 0 & 0 & 0 & 0 & 0 \\
3 & 1 & 0 & -3 & 0 & 5 & -5 & 1 & 0 & 0 & 0 & 0 \\
2 & 1 & 0 & -3 & 0 & 4 & -4 & 0 & 1 & 0 & 0 & 0 \\
2 & 0 & 0 & -3 & 0 & 4 & -4 & 0 & 0 & 1 & 0 & 0 \\
3 & 1 & 0 & -4 & 0 & 6 & -6 & 0 & 0 & 0 & 1 & 0 \\
4 & 2 & 0 & -6 & 0 & 9 & -9 & 0 & 0 & 0 & 0 & 1
\end{array}
$}\right]$  
& 
$\begin{array}{c}
 \KT{12}/I \text{ with $I$ gen.~by }
 \\
  4T_{4}^3T_{6}^2T_{9}T_{10}+T_{1}^2T_{7}T_{12}-T_{3}^2T_{8}
\end{array}
$
\\\\\hline \\
$10$ & $BD_3$ & $\ZZ^{5}$ & $\left[\mbox{\tiny $
\begin{array}{rrrrrrrrr}
-3 & -3 & 0 & -2 & 2 & 0 & 0 & 0 & 0 \\
1 & 1 & 0 & 0 & -1 & 1 & 0 & 0 & 0 \\
2 & 3 & -1 & 2 & -3 & 0 & 2 & 0 & 0 \\
-2 & -5 & 2 & -3 & 5 & 0 & -5 & 1 & 0 \\
-2 & -4 & 2 & -3 & 4 & 0 & -4 & 0 & 1
\end{array}
$}\right]$  
& 
$\begin{array}{c}
 \KT{9}/I \text{ with $I$ gen.~by }
 \\
 4T_{4}^3T_{6}^2T_{9}+T_{1}^2T_{7}-T_{2}^2T_{8}
 \\
\end{array}
$
\\\\
\myhline
\end{longtable}
\endgroup
\end{theorem}

\begin{proof}
{\em Cases where $G\ne A_4$: }
 The Cox rings of the $X_0 = \KK^3/G$ have been presented in Proposition~\ref{prop:coxX0}. We then obtain the Cox ring of a resolution $X\to X_0$ using our implementation~\ref{sec:implementation} of Algorithm~\ref{algo:resolve}.
  The verifications of the primality of the variables $T_i$ are done directly computationally, the smoothness tests are done with Algorithm~\ref{algo:issmooth} as implemented in~\texttt{MDSpackage}~\cite{HaKe}.
  
  {\em Case $G=A_4$: }
To obtain the Cox ring of a resolution of the quotient singularity for $G = A_4$ one needs to change the presentation of the Cox ring of $\KK^3/G$. The one shown in Proposition~\ref{prop:coxX0} comes from the generating set of $\KK[x, y, z]^{[G,G]}$
$$T_1 = x^2+y^2+z^2, \quad T_2 = xyz, \quad T_3 = x^2+\varepsilon y^2 + \varepsilon^2 z^2, \quad T_4 = x^2+\varepsilon^2 y^2 + \varepsilon z^2$$
composed of eigenvectors of $G/[G,G]$ action (where $\varepsilon$ is the 3rd root of unity). 
For the resolution algorithm to work (compare Example~\ref{ex:nonex}), 
we want the generating set to satisfy~\cite[Condition~3.6]{DobuMa}, developed in the more general context including the present case in~\cite{Ya}. Roughly speaking, it determines when a generating set of $\Cox(X_0)$ is suitable for being extended to the Cox ring of its resolution. 
Thus we add to the generating set elements
$$
T_5\, =\, T_1^2 - T_3T_4, \qquad 
T_6\, =\, T_3^2-T_1T_4, 
\qquad T_7 \,=\, T_4^2-T_1T_3,
$$
which are obstructions for this condition for the original generating set. 
We obtain 
a presentation $\Cox(X_0) = \KT{7}/I_0$ 
where generators for $I_0$ and the $\ZZ/3\ZZ$-grading matrix are
\begingroup
\setlength{\tabcolsep}{2pt}
\arraycolsep=3.2pt
\footnotesize
\begin{gather*}
T_4T_5+T_3T_6+T_2T_7,\quad
T_3T_5+T_2T_6+T_4T_7,\quad
T_3^2-T_2T_4-T_7,\\
T_2T_3-T_4^2+T_6,\quad 
T_2^2-T_3T_4-T_5,\quad
27T_1^2-T_2T_5-T_4T_6-T_3T_7,\\
\left[\mbox{\tiny $
\begin{array}{rrrrrrr}
\b 0 & \b 0 & \b 2 & \b 1 & \b 0 & \b 2 & \b 1
\end{array}
$}\right].
\end{gather*}
\endgroup

The ring in the table is obtained via toric ambient modification from the Cox ring of $\KK^3/G$, presented as above.
More precisely, we apply Algorithm~\ref{algo:bruteforce} (to be explained in Subsection~\ref{subsection:smallres}) 
with the choice of the three vectors
$$
v_1\, :=\, (0, 0, 2, 1, 0, 1, 1),\qquad
v_2\, :=\, (0, 0, 2, 1, 0, 1, 2)\qquad
v_3\, :=\, (1, 0, 3, 1, 1, 2, 3).
$$
All the necessary verifications,
in particular the smoothness test, are successful.  
\end{proof}

\begin{example}[Non-example]\label{ex:nonex}
Note that for $G=A_4$, 
if we apply Algorithm~\ref{algo:resolve} 
directly to the presentation of $\Cox(X_0)$  given in case $7$ of Proposition~\ref{prop:coxX0},
we only obtain the Cox ring of a modification $X\to X_0$
where $X$ is not smooth.
More precisely, it computes the $\ZZ^8$-graded Cox ring and degree matrix

\begingroup
\setlength{\tabcolsep}{2pt}
\arraycolsep=3.2pt
\footnotesize
\begin{longtable}{cc}
$\begin{array}{c}
 \KT{12}/I \text{ with $I$ gen.~by }
 \\
 T_{3}^3T_{6}^2T_{7}T_{8}^2T_{11}+T_{4}^3T_{6}T_{7}^2T_{9}^2T_{12}\\
 -3T_{1}T_{3}T_{4}T_{6}T_{7}T_{8}T_{9}T_{11}T_{12}+T_{1}^3T_{8}T_{9}T_{11}^2T_{12}^2\\
 -27T_{2}^2T_{10}
\end{array}
$
&
 $\left[\mbox{\tiny $
\begin{array}{rrrrrrrrrrrr}
-2 & -3 & -2 & -2 & 1 & 0 & 0 & 0 & 0 & 0 & 0 & 0 \\
0 & 0 & -2 & -1 & 0 & 3 & 0 & 0 & 0 & 0 & 0 & 0 \\
0 & 0 & 1 & 0 & 0 & -2 & 1 & 0 & 0 & 0 & 0 & 0 \\
-1 & -1 & 0 & 0 & 0 & -2 & 0 & 1 & 0 & 0 & 0 & 0 \\
-1 & -1 & 2 & 0 & 0 & -4 & 0 & 0 & 1 & 0 & 0 & 0 \\
-1 & -2 & 1 & 0 & 0 & -3 & 0 & 0 & 0 & 1 & 0 & 0 \\
-2 & -2 & 1 & 0 & 0 & -4 & 0 & 0 & 0 & 0 & 1 & 0 \\
-2 & -2 & 2 & 0 & 0 & -5 & 0 & 0 & 0 & 0 & 0 & 1
\end{array}
$}\right]$.
\end{longtable}
\endgroup
\qed
\end{example}

\begin{remark}
 Using these methods and its implementation, one can directly go to higher group orders~$|G|$. Since the number of different representations to consider grows quickly for isomorphism types of~$G$ with higher order, we end the table at order~$12$.
\end{remark}

\subsection{Properties of quotients and their resolutions}\label{subsect:properties-resolutions}
We discuss certain geometric properties of the quotient singularities and their resolutions corresponding to the Cox rings from Theorem~\ref{thm:dim3}. At first, we turn to torus actions on quotient spaces: we consider the relation between the form of the Cox ring and the existence of a torus action on~$\KK^n/G$.

\begin{proposition}
Let $V$ be an affine space which is a direct sum of $n$ representations of a given group $G$. Then $V$ admits an action of $(\KK^*)^n$ which commutes with the action of~$G$.
\end{proposition}

\begin{proof}
Let $(x_{i,1}\ldots,x_{i,n_i})$ be the coordinates of the $i$-th representation (of dimension~$d_i$).
The following action commutes with the action of~$G$.
\begin{gather*}
(\KK^*)^n\,\times\, V\,\to\,V,
\qquad
\begin{array}{cl}
&(t_1,\ldots,t_n)\cdot(x_{1,1}\ldots,x_{1,d_1},\ldots,x_{n,1}\ldots,x_{n,d_n})
\\
= &(t_1x_{1,1}\ldots,t_1x_{1,d_1},\ldots,t_nx_{n,1}\ldots,t_nx_{n,d_n})
\end{array}
\end{gather*}
\end{proof}

\begin{corollary}
A quotient of $\KK^3$ by a direct sum of a $2$-dimensional and a $1$-dimensional representation is a $T$-variety of complexity one.
\end{corollary}

\begin{remark}
\label{rem:cplx1}
All representations in the table in Proposition~\ref{prop:reps} except the case of $A_4$ are direct sums of two irreducible representations, as it can be easily seen from the generating matrices. In particular, all varieties $X_0$ in Proposition~\ref{prop:coxX0} except possibly for the $A_4$ case are $T$-varieties of complexity one.
\end{remark}

Thus, the only $3$-dimensional quotient singularities which are possibly not $T$-varieties of complexity one correspond to irreducible representations.
The representation of $A_4$ in Proposition~\ref{prop:reps} is irreducible and we expect that there is no $(\KK^*)^2$ action.

Note that although we do not give a proof that for 3-dimensional quotient singularities with the $(\KK^*)^2$-action (as in Remark~\ref{rem:cplx1}) Algorithm~\ref{algo:resolve} will always compute the Cox ring of a resolution and that it will have just one trinomial relation, it  seems natural to predict such behaviour. This is because of similar results in slightly different settings. In~\cite[Thm~3.4.4.9]{ArDeHaLa} the authors construct the Cox ring for any complete $T$-variety of complexity one and determine its single relation. Recently, an analogous description of the Cox rings of affine $\KK^*$-surfaces was provided in~\cite[Sect.~3]{HaWro}.

\begin{remark}
By the previous discussion, it would be interesting to continue the list of results in Theorem~\ref{thm:dim3}  with irreducible, $3$-dimensional representations. However, the smallest cases of such representations (one of order $21$, two of order $24$ and two of order $27$) are at the moment computationally out of reach on our machines. 
\end{remark}

We now turn to properties of the resolutions found in Theorem~\ref{thm:dim3}.
Recall that a resolution of singularities $\pi \colon X \rightarrow X_0$ is called {\em crepant} if $K_X = \pi^*K_{X_0}$. For surface quotient singularities for $G \subseteq \SL(2)$, i.e. du Val singularities, crepant resolutions are the minimal ones: the special fiber is a tree of smooth rational curves dual to a Dynkin diagram $A_n$, $D_n$, $E_6$, $E_7$ or $E_8$. The relation between the structure of~$G$ (its conjugacy classes or irreducible representations) and the shape of the diagram of the resolution of $\KK/G$ was noticed by McKay.
The postulated relation between the geometry of crepant resolutions of (Gorenstein) quotient singularities and the structure of the group is called the {\em McKay correspondence}. It has been studied in several special cases and in different formulations. In particular, it is proved in dimension 3, see e.g.~\cite{ItoReid}, for symplectic singularities in dimension 4, see~\cite{KaledinMcKay}, and a weak version for any $G \subset \SL(n)$ (the equality of the dimension of cohomology space and the number of conjugacy classes of $G$) is due to Batyrev~\cite{BatyrevMcKay}.

It is a natural question to ask how good the resolutions obtained using Algorithm~\ref{algo:resolve} are and what properties to expect of them. In particular, we would like to know whether they are crepant for $G \subset \SL(n)$. To test this property for the $3$-dimensional results in Theorem~\ref{thm:dim3} we can use the McKay correspondence.

\begin{remark}
Among groups listed in Proposition~\ref{prop:reps} the cases $1$, $2$, $3$, $5$, $6$, $7$ and $8$ are in $\SL(3)$.
\end{remark}

\begin{proposition}\label{prop:crepant}
In Theorem~\ref{thm:dim3}, 
the Cox rings of resolutions of singularities in cases 1 ($S_3$) and 5 ($D_{10}$), are not the Cox rings of crepant resolutions. The resolutions in cases 2 ($D_8$), 3 ($Q_8$), 6 ($D_{12}$), 7 ($A_4$)
and 8 ($BD_3$) are crepant.
\end{proposition}

\begin{proof}
By~\cite[Lem.~2.11]{81resolutions}, the class group $\Cl(X)$ of a resolution $\pi \colon X \rightarrow \CC^n/G$ of a quotient singularity is a free group and its rank~$m$ is equal to the number of irreducible components of the exceptional divisor. By the McKay correspondence in dimension~$n=3$ we know that~$m$ is the number of conjugacy classes of junior elements in~$G$ where an element $g \in G$ is {\em junior} if 
$$
{\rm age}(g)\ =\ 1
\qquad\text{where}\qquad
{\rm age}(g)\ :=\ \frac{1}{r}\sum_{k=1}^n a_k
$$ 
with~$a_k$ coming from the exponents of eigenvalues $e^{\frac{2\pi i a_k}{r}}$ of~$g$, see e.g.~\cite{ReidBourbaki}. 
Thus the conjugacy classes of junior elements can be determined with simple computations (e.g. using \texttt{GAP}~\cite{GAP4}), and the results are as follows:

\begingroup
\footnotesize
\def\arraystretch{0.5}
\begin{longtable}{cccccccc}
\myhline
case number & $1$ & $2$ & $3$ & $5$ & $6$ & $7$ & $8$ \\
\myhline
junior classes & $2$ & $4$ & $4$ & $3$ & $5$ & $3$ & $5$\\
\myhline
\end{longtable}
\endgroup

Comparing them with the rank of $\Cl(X)$ given in Theorem~\ref{thm:dim3} we obtain that the resolutions in cases $1$ and $5$ have too many components of the exceptional divisor to be crepant.

In the remaining cases obtained resolutions are crepant: since any crepant divisor appears necessarily on any resolution (see e.g.~\cite[2.3]{ItoReid}), it is enough to show that the number of exceptional divisors is equal to the number of junior classes.
\end{proof}

Note that in dimension $3$ there is one Cox ring corresponding to all crepant resolutions, because they are all birationally equivalent and flops preserves smoothness in dimension~$3$.

\begin{remark}\label{rem:gitcones}
We can describe the geometry of all crepant resolutions for groups $2$, $3$, $6$, $7$ and $8$ of Proposition~\ref{prop:reps} by computing GIT quotients of the spectrum of their Cox rings. To determine the GIT fan describing the variation of the quotients, we use~\cite{Ke}.

\begingroup
\footnotesize
\def\arraystretch{0.5}
\begin{longtable}{cccccccc}
\myhline
case number & $2$ & $3$ & $6$ & $7$ & $8$\\
\myhline
max.~GIT-cones within $\mov(X)$ & $9$ & $1$ & $16$ & $5$ & $1$\\
\myhline
\end{longtable}
\endgroup

Actually, the cases 3 and 8 are products of a representation of a finite group in $\SL(2)$ and a 1-dimensional trivial representation, so the resolutions will be just products of minimal resolutions of du Val singularities by $\CC$. Hence, there is just one crepant resolution, i.e., just one chamber in the GIT fan restricted to the cone of movable divisor classes~$\Mov(X)$.
\end{remark}

\subsection{Smaller resolutions}
\label{subsection:smallres}
In Algorithm~\ref{algo:resolve},
we computed Cox rings of resolutions as proper transforms under toric resolutions. 
The latter was obtained by intersection of the tropical variety with the affine toric ambient variety of $X_0$.
This may yield unnecessarily many new rays and therefore bigger resolutions.
In this subsection, we shortly describe an immediate alternative version that works with subsets of these rays in a brute force manner.
A similar idea was used in~\cite[Sect.~3]{HaKe}.
We can then improve two resolutions of Theorem~\ref{thm:dim3}.

\begin{algorithm}
 \label{algo:bruteforce}
 {\em Input: } as in the first steps of Algorithm~\ref{algo:resolve}: the Cox ring $\KT{s}/I_0$ of $X_0$, $\sigma=P_0(\QQQ^s)$, a fan $\Upsilon$ with support $\Trop(X_0)$,  $\Sigma = \{\sigma\}\cap \Upsilon$.
 \begin{itemize}
 \item 
 Resolve $\Sigma$, i.e., determine a set $\mathcal V\subseteq \QQ^n$ of primitive vectors such that the stellar subdivision of $\Sigma$ at $\mathcal V$ is regular.
  \item 
  Insert into $\mathcal V$ primitive generators for the rays of $\Sigma$ that are not rays of~$\sigma$.
  \item For each $k=1,\ldots,|\mathcal V|$, do:
  \begin{itemize}
   \item 
   For each $v_{i_1},\ldots,v_{i_k}\in\mathcal V$, do:
   \begin{itemize}
      \item Perform the steps of Algorithm~\ref{algo:resolve} starting from line~$6$ with $v_1,\ldots,v_m$ replaced by $v_{i_1},\ldots,v_{i_k}$ and $\Sigma$ replaced by the stellar subdivision of $\sigma$ at these vectors. 
      Stop, if the verification is positive.
      \end{itemize}
  \end{itemize}
 \end{itemize}
 {\em Output (if provided): } the Cox ring $\Cox(X)=\KT{r}/I$ of a resolution $X\to X_0$.
\end{algorithm}

\begin{proposition}\label{prop:3dimcrepant}
In the setting of Theorem~\ref{thm:dim3},
the following are Cox rings of crepant resolutions.

\begingroup
\setlength{\tabcolsep}{2pt}
\arraycolsep=1.2pt
\footnotesize
\def\arraystretch{0.5}
\begin{longtable}{ccccc}
\myhline
No.  & $G$ & $Cl(X)$ & degree matrix & $\Cox(X)$\\
\myhline
\\
$1$ & $S_3$ & $\ZZ^2$ & $\left[\mbox{\tiny $
\begin{array}{rrrrrr}
1 & 1 & 0 & 0 & -2 & 0 \\
0 & -1 & -1 & -1 & 0 & 1
\end{array}
$}\right]$  
& 
$\begin{array}{c}
 \KT{6}/I \text{ with $I$ gen.~by }
 \\
4T_3^3T_6-T_2^2T_5-27T_4^2
\end{array}
$
\\\\\hline \\
$5$ & $D_{10}$ & $\ZZ^3$ & $\left[\mbox{\tiny $
\begin{array}{rrrrrrr}
    1 & 1 & 0 & 0 & -2 & 0 & 0 \\
    0 & -1 & -1 & -1 & 0 & 1 & 0 \\
    0 & -1 & 0 & -1 & 0 & -1 & 1
    \end{array}
$}\right]$  
& 
$\begin{array}{c}
 \KT{7}/I \text{ with $I$ gen.~by }
 \\
 {4}T_{3}^{5}T_{6}^{3}T_{7}+T_{2}^{2}T_{5}-T_{4}^{2}
\end{array}
$
\\\\\myhline \\
\end{longtable}
\endgroup
\end{proposition}

\begin{proof}
 In case $1$,  the fan $\Sigma_0$ 
  of the ambient affine toric variety has the maximal cone $\sigma$ generated by
 $$
 (1, 0, 0, 0),\quad (1, 2, 0, 0),\quad (0, 0, 1, 0),\quad (0, 0, 0, 1)\ \in \QQ^4.
 $$
 Algorithm~\ref{algo:bruteforce} then  stellarly subdivides $\Sigma_0={\rm fan}(\sigma)$ at
  $(1, 1, 0, 0)$ and $(1, 2, 1, 1)\in \QQ^4$. The verifications succeed.
Similarly, in case $5$, the cone $\sigma$ is generated by
$$
(1, 0, 0, 0),\quad (1, 2, 0, 0),\quad (0, 0, 1, 0),\quad (0, 0, 0, 1)\ \in\QQ^4.
$$
Again, Algorithm~\ref{algo:bruteforce}  stellarly subdivides $\Sigma_0$ at the vectors
  $(1, 1, 0, 0)$, $(1, 2, 1, 1)$ and $(2, 4, 1, 2)\in \QQ^4$.
  
 Both resolutions are crepant, because the number of exceptional divisors is the same as the number of junior conjugacy classes in $G$, compare the end of the proof of~\ref{prop:crepant}. 
\end{proof}

\begin{remark}
As in Remark~\ref{rem:gitcones}, we compute the variation of GIT-quotients for the two cases of Proposition~\ref{prop:3dimcrepant}:

\begingroup
\footnotesize
\def\arraystretch{0.5}
\begin{longtable}{ccc}
\myhline
case number & $1$ & $5$\\
\myhline
max.~GIT-cones within $\mov(X)$ & $2$ & $3$ \\
\myhline
\end{longtable}
\endgroup

\end{remark}

\begin{remark}
Applying Algorithm~\ref{algo:bruteforce} 
to the other cases of Theorem~\ref{thm:dim3}
does not yield better resolutions in the sense of smaller Picard number or fewer generators. Note that these cases are precisely the ones which are not in~$\SL(3)$.
\end{remark}

\section{Two $4$-dimensional examples}
\label{sec:4dimex}

In this section, we present two $4$-dimensional examples.
In dimension~4 much less is known about the resolutions of quotient singularities. In particular, crepant resolutions do not always exist, and the McKay correspondence has been proved just for the symplectic case by Kaledin, see~\cite{KaledinMcKay}. Hence it is a very appropriate setting for computational experiments with constructing resolutions via Cox rings.
Moreover, it is interesting also from the point of view of the Cox ring theory: while in dimension~3 all our cases except for one can be defined with a single trinomial relation, here the ring structure will be more complicated. 

The groups given below are complex symplectic, hence we work here with symplectic quotient singularities. Both examples have also been treated in~\cite{DobuMa} where the Cox rings of symplectic resolutions was constructed.

\begin{proposition}
\label{prop:4dim}
Consider the $4$-dimensional quotient singularities $\KK^4/G$ for the two small groups $G\subseteq \GL(4)$

\begingroup
\footnotesize
\def\arraystretch{0.5}
\begin{longtable}{cl}
\myhline
$G$ & gen.s in $\GL(4)$\\
\myhline\\
$S_3$ & $\left[\mbox{\tiny $
\begin{array}{rrrr}
0 & -1 & 0 & 0\\ 
1 & -1 & 0 & 0\\
0 & 0 & 0 & -1\\
0 & 0 & 1 & -1
\end{array}
$}\right],\quad 
 \left[\mbox{\tiny $
\begin{array}{rrrr}
-1 & 1 & 0 & 0\\ 
0 & 1 & 0 & 0\\
0 & 0 & -1 & 1\\
0 & 0 & 0 & 1
 \end{array}
 $}\right].
$
\\\\\hline\\
$D_8$ & 
$\left[\mbox{\tiny $
\begin{array}{rrrr}
0 & -1 & 0 & 0 \\
1 & 0 & 0 & 0 \\
0 & 0 & 0 & -1\\
0 & 0 & 1 & 0
 \end{array}
 $}\right],\quad 
 \left[\mbox{\tiny $
\begin{array}{rrrr}
1 & 0 & 0 & 0 \\
0 & -1 & 0 & 0 \\
0 & 0 & 1 & 0 \\
0 & 0 & 0 & -1
 \end{array}
 $}\right].
$
\\\\
\myhline
\end{longtable}
\endgroup
The Cox ring of the quotient space $X_0 := \KK^n/G$, the degree matrix and the class group of $X_0$ are as follows.

\begingroup
\footnotesize
\def\arraystretch{0.6}
\begin{longtable}{cccc}
\myhline
$G$  & \multicolumn{2}{c}{$Cl(X_0)$, degree matrix and $\Cox(X_0)$}\\
\myhline\\
$S_3$ &  $\ZZ/2\ZZ$ & $\left[\mbox{\tiny $
\begin{array}{rrrrrrrrrrrr}
\b 1 & \b 1 & \b 1 & \b 1 & \b 1 & \b 0 & \b 0 & \b 0 & \b 0 & \b 0 & \b 0 & \b 0
\end{array}
$}\right]$  
\\\\
\multicolumn{3}{c}{
$\begin{array}{cc}
 \KT{12}/I \text{ with $I$ gen.~by }
 \\[1ex]
3T_{9}T_{10}-T_{7}T_{11}+T_{6}T_{12}, &
3T_{5}T_{10}-2T_{4}T_{11}+T_{3}T_{12},\\
3T_{6}T_{8}+T_{9}T_{11}-T_{7}T_{12}, &
3T_{4}T_{8}-T_{2}T_{11}-2T_{5}T_{12},\\
3T_{3}T_{8}-3T_{2}T_{10}-T_{5}T_{11}-T_{4}T_{12}, &
T_{5}T_{7}-2T_{4}T_{9}+T_{1}T_{12},\\
T_{4}T_{7}-2T_{3}T_{9}+T_{1}T_{11}, &
T_{2}T_{7}-3T_{1}T_{8}+2T_{5}T_{9},\\
T_{5}T_{6}-T_{3}T_{9}+T_{1}T_{11}, &
2T_{4}T_{6}-T_{3}T_{7}+3T_{1}T_{10},\\
T_{2}T_{6}+T_{4}T_{9}-T_{1}T_{12}, &
T_{1}T_{5}+6T_{9}T_{11}-3T_{7}T_{12},\\
T_{4}^2-T_{3}T_{5}-3T_{11}^2+9T_{10}T_{12}, &
T_{2}T_{4}+T_{5}^2+9T_{8}T_{11}-3T_{12}^2,\\
T_{1}T_{4}+3T_{7}T_{11}-6T_{6}T_{12}, &
T_{2}T_{3}+T_{4}T_{5}+27T_{8}T_{10}-3T_{11}T_{12},\\
T_{1}T_{3}+9T_{7}T_{10}-6T_{6}T_{11}, &
T_{1}T_{2}+9T_{7}T_{8}-6T_{9}T_{12},\\
T_{1}^2+3T_{7}^2-12T_{6}T_{9}, &
4T_{9}^3-T_{2}^2-27T_{8}^2,\\
2T_{7}T_{9}^2+T_{2}T_{5}-9T_{8}T_{12}, &
4T_{6}T_{9}^2-T_{5}^2-3T_{12}^2,\\
2T_{1}T_{9}^2-9T_{5}T_{8}-3T_{2}T_{12}, &
T_{7}^2T_{9}-4T_{6}T_{9}^2-9T_{8}T_{11}+3T_{12}^2,\\
2T_{6}T_{7}T_{9}-T_{4}T_{5}-3T_{11}T_{12}, &
T_{1}T_{7}T_{9}-3T_{2}T_{11}-3T_{5}T_{12},\\
4T_{6}^2T_{9}-T_{3}T_{5}-6T_{11}^2+9T_{10}T_{12}, &
2T_{1}T_{6}T_{9}+3T_{5}T_{11}-3T_{4}T_{12},\\
T_{7}^3-4T_{6}T_{7}T_{9}-27T_{8}T_{10}+3T_{11}T_{12}, &
T_{6}T_{7}^2-4T_{6}^2T_{9}+3T_{11}^2-9T_{10}T_{12},\\
T_{1}T_{7}^2-4T_{1}T_{6}T_{9}-9T_{3}T_{8}-3T_{5}T_{11}+6T_{4}T_{12}, &
2T_{6}^2T_{7}-T_{3}T_{4}-9T_{10}T_{11},\\
T_{1}T_{6}T_{7}+3T_{4}T_{11}-3T_{3}T_{12}, &
4T_{6}^3-T_{3}^2-27T_{10}^2,\\
2T_{1}T_{6}^2+9T_{4}T_{10}-3T_{3}T_{11} &
\end{array}
$
}
\\\\\hline\\
$D_8$ &  $\ZZ/2\ZZ\times \ZZ/2\ZZ$ & $\left[\mbox{\tiny $
\begin{array}{rrrrrrrrrr}
\b 1 & \b 1 & \b 1 & \b 1 & \b 1 & \b 1 & \b 0 & \b 0 & \b 0 & \b 0\\
\b 1 & \b 1 & \b 1 & \b 0 & \b 0 & \b 0 & \b 1 & \b 0 & \b 0 & \b 0
\end{array}
$}\right]$  
\\\\
\multicolumn{3}{c}{
$\begin{array}{cc}
 \KT{10}/I \text{ with $I$ gen.~by }
 \\[1ex]
T_{5}T_{8}+T_{4}T_{9}-2T_{6}T_{10} &
T_{3}T_{8}+T_{1}T_{9}-T_{2}T_{10},\\
T_{7}^2-T_{8}T_{9}+T_{10}^2,&
T_{6}T_{7}+2T_{1}T_{9}-T_{2}T_{10},\\
T_{5}T_{7}+T_{2}T_{9}-2T_{3}T_{10},&
T_{4}T_{7}-T_{2}T_{8}+2T_{1}T_{10},\\
2T_{3}T_{7}-T_{6}T_{9}+T_{5}T_{10},&
T_{2}T_{7}-T_{4}T_{9}+T_{6}T_{10},\\
2T_{1}T_{7}+T_{6}T_{8}-T_{4}T_{10},&
T_{4}T_{5}-T_{6}^2+T_{8}T_{9}-T_{10}^2,\\
T_{2}T_{5}-2T_{3}T_{6}+T_{7}T_{9},&
2T_{1}T_{5}-T_{2}T_{6}+T_{7}T_{10},\\
T_{3}T_{4}-T_{1}T_{5}-T_{7}T_{10},&
T_{2}T_{4}-2T_{1}T_{6}-T_{7}T_{8},\\
4T_{3}^2+T_{5}^2-T_{9}^2,&
2T_{2}T_{3}+T_{5}T_{6}-T_{9}T_{10},\\
4T_{1}T_{3}+T_{6}^2-T_{10}^2,&
T_{2}^2+T_{6}^2-T_{8}T_{9},\\
2T_{1}T_{2}+T_{4}T_{6}-T_{8}T_{10},&
4T_{1}^2+T_{4}^2-T_{8}^2
\end{array}
$
}

\\\\
\myhline
\end{longtable}
\endgroup
\end{proposition}

\begin{proof}
 This is again an application of Algorithm~\ref{algo:coxquotsing}. 
\end{proof}

We now compute the Cox rings of a resolution for the two cases 
$X_0=\KK^4/S_3$ and $X_0=\KK^4/D_8$ presented in the previous proposition.
In Proposition~\ref{prop:4dim1}, we present a resolution for the $D_8$-case thereby retrieving a result of Grab and the first author~\cite[Prop.~5.15]{DobuMa}; Proposition~\ref{prop:4dimS3} treats the case of~$S_3$.

\begin{remark}
In the case of $G=D_8$ in Proposition~\ref{prop:4dim},
applying Algorithm~\ref{algo:resolve} to $X_0=\KK^4/G$ without changes, we obtain a modification $X\to X_0$ with $\Cl(X)=\ZZ^{16}$
and the Cox ring $\Cox(X) = \KT{26}/I$ with $28$ generators for $I$.
Due to this size, we could not verify smoothness on our machines.  
Similarly, the smoothness tests in the case $G=S_3$ were computationally infeasible.
\end{remark}

It turns out, that after some changes in Algorithm~\ref{algo:bruteforce}, we obtain the Cox rings of symplectic resolutions for $S_3$ and~$D_8$. Note that for a symplectic quotient singularity a crepant resolution and a symplectic resolution is the same, see~\cite[Thm~2.5]{Verbitsky_crepant}.

\begin{proposition}\label{prop:4dim1}
In the case of $G=D_8$ in Proposition~\ref{prop:4dim},
applying a variant of Algorithm~\ref{algo:resolve} to $X_0=\KK^4/G$, 
we obtain a resolution $X\to X_0$ with $\Cl(X)=\ZZ^{2}$
and the Cox ring is $\Cox(X) = \KT{12}/I$ where generators for $I$ and the degree matrix are

\begin{center}
\allowdisplaybreaks
\footnotesize
$
\begin{array}{rl}
T_{5}T_{8}+T_{4}T_{9}-2T_{6}T_{10}, &
T_{3}T_{8}+T_{1}T_{9}-T_{2}T_{10},\\
T_{2}T_{5}-2T_{3}T_{6}+T_{7}T_{9}, &
2T_{1}T_{5}-T_{2}T_{6}+T_{7}T_{10},\\
2T_{3}T_{4}-T_{2}T_{6}-T_{7}T_{10}, &
T_{2}T_{4}-2T_{1}T_{6}-T_{7}T_{8},\\
T_{7}^2T_{12}-T_{2}^2+4T_{1}T_{3}, &
T_{6}T_{7}T_{12}+2T_{1}T_{9}-T_{2}T_{10},\\
T_{5}T_{7}T_{12}+T_{2}T_{9}-2T_{3}T_{10}, &
T_{4}T_{7}T_{12}-T_{2}T_{8}+2T_{1}T_{10},\\
T_{4}T_{5}T_{12}-T_{6}^2T_{12}+T_{8}T_{9}-T_{10}^2, &
T_{7}^2T_{11}+T_{4}T_{5}-T_{6}^2,\\
2T_{3}T_{7}T_{11}-T_{6}T_{9}+T_{5}T_{10}, &
T_{2}T_{7}T_{11}-T_{4}T_{9}+T_{6}T_{10},\\
2T_{1}T_{7}T_{11}+T_{6}T_{8}-T_{4}T_{10}, &
4T_{3}^2T_{11}+T_{5}^2T_{12}-T_{9}^2,\\
2T_{2}T_{3}T_{11}+T_{5}T_{6}T_{12}-T_{9}T_{10}, &
4T_{1}T_{3}T_{11}+T_{4}T_{5}T_{12}+T_{8}T_{9}-2T_{10}^2,\\
2T_{1}T_{2}T_{11}+T_{4}T_{6}T_{12}-T_{8}T_{10}, &
4T_{1}^2T_{11}+T_{4}^2T_{12}-T_{8}^2
\end{array}
$
\\[1ex]
$
\left[\mbox{\tiny $
\begin{array}{rrrrrrrrrrrr}
-1 & -1 & -1 & 0 & 0 & 0 & -1 & 0 & 0 & 0 & 2 & 0 \\ 
1 & 1 & 1 & -1 & -1 & -1 & 0 & 0 & 0 & 0 & -2 & 2 
\end{array}
$}
\right].
$
\end{center}
After a suitable linear change of coordinates one sees that this ring is isomorphic to the Cox ring of symplectic, i.e. crepant, resolutions of considered representation of~$D_8$ computed in~\cite[Sect.~5]{DobuMa}.
\end{proposition}

\begin{proof}
This is an application of Algorithm~\ref{algo:bruteforce}
with the following modifications (similar to~\cite[Sect.~3]{HaKe}):
 There are three minimal elements $\sigma_1,\sigma_2,\sigma_3$ of the set consisting of all singular $P_0(\gamma_0^*)$ such that $\gamma_0\text{ is an $\mathfrak F$-face}$.
 We form the cone-wise intersections
 $\Sigma_i = \{\sigma_i\}\cap \Upsilon$ and resolve the fans: $\Sigma_i'\to \Sigma_i$.
 Denoting by ${\rm primit}(\Xi^{(1)})$ the primitive generators of the rays of a fan or cone $\Xi$, 
 we have
 $$
 \bigcup_{i=1}^3{\rm primit}\left(\Sigma_i'^{(1)}\right)\setminus {\rm primit}\left(\sigma_i^{(1)}\right)
 \ =\ 
 \left\{\!\!
 \mbox{\footnotesize $
 \begin{array}{c}
 (2, 2, 1, 0, 0, 0, 1, 0, 0, 0),\\
 (2, 2, 2, 1, 1, 1, 1, 0, 0, 0)
  \end{array}
  $}\!\!
  \right\}
  \ =:\ \{v_1,v_2\}.
 $$
The remaining steps of the algorithm deliver the result.
All needed verifications, in particular the smoothness test, succeed. 
\end{proof}

\begin{proposition}\label{prop:4dimS3}
In the case of $G=S_3$ in Proposition~\ref{prop:4dim} we obtain a resolution $X\to X_0$ with $\Cl(X)=\ZZ$ and the Cox ring is $\Cox(X) = \KT{13}/I$ where generators for $I$ and the degree matrix are 

\begin{center}
\allowdisplaybreaks
\footnotesize
$
\begin{array}{rl}
    3T_{9}T_{10}-T_{7}T_{11}+T_{6}T_{12}, &
    3T_{5}T_{10}-2T_{4}T_{11}+T_{3}T_{12},\\
    3T_{6}T_{8}+T_{9}T_{11}-T_{7}T_{12}, &
    3T_{4}T_{8}-T_{2}T_{11}-2T_{5}T_{12},\\
    3T_{3}T_{8}-3T_{2}T_{10}-T_{5}T_{11}-T_{4}T_{12}, &
    T_{5}T_{7}-2T_{4}T_{9}+T_{1}T_{12},\\
    T_{4}T_{7}-2T_{3}T_{9}+T_{1}T_{11}, &
    T_{2}T_{7}-3T_{1}T_{8}+2T_{5}T_{9},\\
    T_{5}T_{6}-T_{3}T_{9}+T_{1}T_{11}, &
    2T_{4}T_{6}-T_{3}T_{7}+3T_{1}T_{10},\\
    T_{2}T_{6}+T_{4}T_{9}-T_{1}T_{12}, &
    T_{1}T_{5}T_{13}+6T_{9}T_{11}-3T_{7}T_{12},\\
    T_{4}^2T_{13}-T_{3}T_{5}T_{13}-3T_{11}^2+9T_{10}T_{12}, & 
    T_{2}T_{4}T_{13}+T_{5}^2T_{13}+9T_{8}T_{11}-3T_{12}^2,\\
    T_{1}T_{4}T_{13}+3T_{7}T_{11}-6T_{6}T_{12}, &
    T_{2}T_{3}T_{13}+T_{4}T_{5}T_{13}+27T_{8}T_{10}-3T_{11}T_{12},\\
    T_{1}T_{3}T_{13}+9T_{7}T_{10}-6T_{6}T_{11}, &
    T_{1}T_{2}T_{13}+9T_{7}T_{8}-6T_{9}T_{12},\\
    T_{1}^2T_{13}+3T_{7}^2-12T_{6}T_{9}, &
    4T_{9}^3-T_{2}^2T_{13}-27T_{8}^2,\\
    2T_{7}T_{9}^2+T_{2}T_{5}T_{13}-9T_{8}T_{12}, &
    4T_{6}T_{9}^2-T_{5}^2T_{13}-3T_{12}^2,\\
    2T_{1}T_{9}^2-9T_{5}T_{8}-3T_{2}T_{12}, &
    T_{7}^2T_{9}-4T_{6}T_{9}^2-9T_{8}T_{11}+3T_{12}^2,\\
    2T_{6}T_{7}T_{9}-T_{4}T_{5}T_{13}-3T_{11}T_{12}, &
    T_{1}T_{7}T_{9}-3T_{2}T_{11}-3T_{5}T_{12},\\
    4T_{6}^2T_{9}-T_{3}T_{5}T_{13}-6T_{11}^2+9T_{10}T_{12}, &
    2T_{1}T_{6}T_{9}+3T_{5}T_{11}-3T_{4}T_{12},\\
    T_{1}^2T_{9}-3T_{2}T_{4}-3T_{5}^2, &
    T_{7}^3-4T_{6}T_{7}T_{9}-27T_{8}T_{10}+3T_{11}T_{12},\\
    T_{6}T_{7}^2-4T_{6}^2T_{9}+3T_{11}^2-9T_{10}T_{12}, &
    T_{1}T_{7}^2-4T_{1}T_{6}T_{9}-18T_{3}T_{8}+9T_{2}T_{10}+9T_{4}T_{12},\\
    2T_{6}^2T_{7}-T_{3}T_{4}T_{13}-9T_{10}T_{11}, &
    T_{1}T_{6}T_{7}+3T_{4}T_{11}-3T_{3}T_{12},\\
    T_{1}^2T_{7}-3T_{2}T_{3}-3T_{4}T_{5}, &
    4T_{6}^3-T_{3}^2T_{13}-27T_{10}^2,\\
    2T_{1}T_{6}^2+9T_{4}T_{10}-3T_{3}T_{11}, &
    T_{1}^2T_{6}-3T_{4}^2+3T_{3}T_{5}
\end{array}
$
\\[1ex]
$
\left[\mbox{\tiny $
\begin{array}{rrrrrrrrrrrrr}
 -1 & -1 & -1 & -1 & -1 & 0 & 0 & 0 & 0 & 0 & 0 & 0 & 2
\end{array}
$}
\right].
$
\end{center}
After a suitable linear change of coordinates one sees that this ring is isomorphic to the Cox ring of symplectic resolution of considered representation of~$S_3$ computed in~\cite[Sect.~4]{DobuMa}.
\end{proposition}

\begin{proof}
This is an application of Algorithm~\ref{algo:bruteforce}; one new ray is added to the fan of the toric ambient variety corresponding to the Cox ring of~$X_0$. To show that the ring is isomorphic to the one given by a generating set in~\cite[Sect.~4]{DobuMa} it is sufficient to compute the relations between these generators and perform the coordinate change to pass from one representation of~$S_3$ to the other. Having this, the smoothness tests are not needed: by~\cite[Prop.~4.4]{DobuMa} we indeed obtain the Cox ring of the symplectic resolution of~$X_0$.
\end{proof}

\bibliographystyle{abbrv}
\bibliography{resolve}

\end{document}

%% file: header.tex
\usepackage{a4}
\usepackage{graphics,graphicx}
\usepackage{amssymb,amsmath}
\usepackage[all]{xy}
\usepackage{longtable}
\usepackage{multirow}
\usepackage{stmaryrd}
\CompileMatrices
\usepackage{hyperref}

\usepackage{tikz}
\usetikzlibrary{calc}
\usetikzlibrary{decorations.markings}
\usetikzlibrary{arrows,shapes,matrix}
\usetikzlibrary{positioning,shapes,shadows,arrows}
\usetikzlibrary{mindmap,trees,calc}
\usetikzlibrary{arrows,shapes,matrix}
\usetikzlibrary{calc}
\usetikzlibrary{arrows,shapes}
\usetikzlibrary{positioning,shapes,shadows,arrows}
\usetikzlibrary{decorations.markings}
\usetikzlibrary{decorations.pathreplacing}
\usepackage{dpfloat}
\usepackage{todonotes}

\newtheorem{theorem}{Theorem}[section]
\newtheorem{lemma}[theorem]{Lemma}
\newtheorem{proposition}[theorem]{Proposition}
\newtheorem{corollary}[theorem]{Corollary}

\theoremstyle{definition}
\newtheorem{notation}[theorem]{Notation}
\newtheorem{definition}[theorem]{Definition}
\newtheorem{example}[theorem]{Example}
\newtheorem{remark}[theorem]{Remark}

\newtheorem{construction}[theorem]{Construction}

\newtheorem{algorithm}[theorem]{Algorithm}

\theoremstyle{remark}

\def\TT{\mathbb{T}}
\def\ZZ{\mathbb{Z}}

\def\PP{\mathbb{P}}
\def\QQ{\mathbb{Q}}

\def\OO{\mathcal O}

\def\CC{\mathcal C}

\def\<{\langle}
\def\>{\rangle}

\def\good{/\!\!/}
\def\cox{\mathcal{R}}

\newcommand{\dual}{\vee}

\newcommand{\mov}{{\rm Mov}}

\newcommand{\BD}{\operatorname{BD}}

\makeatletter
\newcommand{\thickhline}{%
    \noalign {\ifnum 0=`}\fi \hrule height 1pt
    \futurelet \reserved@a \@xhline
}
\makeatother

\renewcommand{\phi}{\varphi}
\renewcommand{\theta}{\vartheta}
\renewcommand{\rho}{\varrho}
\newcommand{\myhline}{\noalign{\global\arrayrulewidth.95pt}\hline
                      \noalign{\global\arrayrulewidth.2pt}}

\def\Trop{{\rm Trop}}

\def\b#1{\overline{#1}}

\def\CC{{\mathbb C}}
\def\KK{\CC}
\def\TT{{\mathbb T}}
\def\ZZ{{\mathbb Z}}

\def\QQ{{\mathbb Q}}
\def\PP{{\mathbb P}}
\def\XX{{\mathbb X}}

\def\Cox{\cox}

\def\id{{\rm id}}

\def\Mov{{\rm Mov}}

\def\Cl{\operatorname{Cl}}

\def\GL{{\rm GL}}

\def\codim{{\rm codim}}

\def\Spec{{\rm Spec}}

\def\lin{{\rm lin}}
\def\Lin{\lin}

\def\SL{{\rm SL}}

\def\spec{\Spec}

\def\QQQ{\QQ_{\geq 0}}
\def\KT#1{\KK[T_1,\ldots,T_{#1}]}

\makeatletter
\newcommand{\vast}{\bBigg@{4}}
\newcommand{\Vast}{\bBigg@{5}}
\makeatother

\usepackage{lmodern}
\makeatletter
\ifcase \@ptsize \relax
  \newcommand{\miniscule}{\@setfontsize\miniscule{4}{5}}
\fi
\makeatother

\makeindex

\definecolor{zielony}{rgb}{0.3, 0.75, 0.2}
\definecolor{czerwony}{rgb}{0.8, 0.2, 0.1}
\definecolor{niebieski}{rgb}{0.3, 0.1, 0.9}

%% file: comp_resolve_revised.bbl
\begin{thebibliography}{10}

\bibitem{ArDeHaLa}
I.~Arzhantsev, U.~Derenthal, J.~Hausen, and A.~Laface.
\newblock {\em Cox rings}, volume 144 of {\em Cambridge Studies in Advanced
  Mathematics}.
\newblock Cambridge University Press, Cambridge, 2014.

\bibitem{ArGa}
I.~V. Arzhantsev and S.~A. Ga{\u\i}fullin.
\newblock Cox rings, semigroups, and automorphisms of affine varieties.
\newblock {\em Mat. Sb.}, 201(1):3--24, 2010.

\bibitem{BatyrevMcKay}
V.~V. Batyrev.
\newblock Non-{A}rchimedean integrals and stringy {E}uler numbers of
  log-terminal pairs.
\newblock {\em J. Eur. Math. Soc. (JEMS)}, 1(1):5--33, 1999.

\bibitem{BoJeSpeStu}
T.~Bogart, A.~N. Jensen, D.~Speyer, B.~Sturmfels, and R.~R. Thomas.
\newblock Computing tropical varieties.
\newblock {\em J. Symbolic Comput.}, 42(1-2):54--73, 2007.

\bibitem{Normaliz}
W.~Bruns, B.~Ichim, T.~R\"omer, R.~Sieg, and C.~S\"oger.
\newblock Normaliz. algorithms for rational cones and affine monoids.
\newblock Available at \url{https://www.normaliz.uni-osnabrueck.de}.

\bibitem{CoLiOSh}
D.~Cox, J.~Little, and D.~O{'}Shea.
\newblock {\em {I}deals, {V}arieties, and {A}lgorithms}.
\newblock Undergraduate Texts in Mathematics. Springer-Verlag, New York, 1992.
\newblock An introduction to computational algebraic geometry and commutative
  algebra.

\bibitem{toricvarieties}
D.~A. Cox, J.~B. Little, and H.~K. Schenck.
\newblock {\em Toric varieties}, volume 124 of {\em Graduate Studies in
  Mathematics}.
\newblock American Mathematical Society, Providence, RI, 2011.

\bibitem{singular}
W.~Decker, G.-M. Greuel, G.~Pfister, and H.~Sch\"onemann.
\newblock {\sc Singular} {4-0-3} --- {A} computer algebra system for polynomial
  computations.
\newblock \url{http://www.singular.uni-kl.de}, 2016.

\bibitem{ker}
F.~Di~Biase and R.~Urbanke.
\newblock An algorithm to calculate the kernel of certain polynomial ring
  homomorphisms.
\newblock {\em Experiment. Math.}, 4(3):227--234, 1995.

\bibitem{DobuMa}
M.~Donten-Bury and M.~Grab.
\newblock {Cox rings of some symplectic resolutions of quotient singularities}.
\newblock 2015.
\newblock Preprint. arXiv:1504.07463.

\bibitem{quotsingcox}
M.~Donten-Bury and S.~Keicher.
\newblock \texttt{quotsingcox.lib} -- a {S}ingular library to compute {C}ox
  rings of quotient singularities, 2016.
\newblock \url{www.mathematik.uni-tuebingen.de/~keicher/quotsingcox/}.

\bibitem{81resolutions}
M.~Donten-Bury and J.~A. Wi\'{s}niewski.
\newblock On 81 symplectic resolutions of a 4-dimensional quotient by a group
  of order 32.
\newblock 2014.
\newblock Preprint. arXiv:1409.4204.

\bibitem{GAP4}
The GAP~Group.
\newblock {\em {GAP -- Groups, Algorithms, and Programming, Version 4.7.5}},
  2014.

\bibitem{Ha2}
J.~Hausen.
\newblock Cox rings and combinatorics. {II}.
\newblock {\em Mosc. Math. J.}, 8(4):711--757, 847, 2008.

\bibitem{HaKe}
J.~Hausen and S.~Keicher.
\newblock A software package for {M}ori dream spaces.
\newblock {\em LMS J. Comput. Math.}, 18(1):647--659, 2015.

\bibitem{HaKeLa}
J.~Hausen, S.~Keicher, and A.~Laface.
\newblock Computing {C}ox rings.
\newblock {\em Math. Comp.}, 85(297):467--502, 2016.

\bibitem{HaWro}
J.~Hausen and M.~Wrobel.
\newblock {Non-complete rational T-varieties of complexity one}.
\newblock 2015.
\newblock Preprint. arXiv:1512.08930.

\bibitem{Hu:diss}
E.~Huggenberger.
\newblock {\em Fano Varieties with Torus Action of Complexity One}.
\newblock PhD thesis, Universit{\"a}t T{\"u}bingen, Wilhelmstr. 32, 72074
  T{\"u}bingen, 2013.

\bibitem{ItoReid}
Y.~Ito and M.~Reid.
\newblock The {M}c{K}ay correspondence for finite subgroups of {${\rm
  SL}(3,\bold C)$}.
\newblock In {\em Higher-dimensional complex varieties ({T}rento, 1994)}, pages
  221--240. de Gruyter, Berlin, 1996.

\bibitem{KaledinMcKay}
D.~Kaledin.
\newblock Mc{K}ay correspondence for symplectic quotient singularities.
\newblock {\em Invent. Math.}, 148(1):151--175, 2002.

\bibitem{Ke}
S.~Keicher.
\newblock Computing the {GIT}-fan.
\newblock {\em Internat. J. Algebra Comput.}, 22(7):1250064, 11, 2012.

\bibitem{Ke:diss}
S.~Keicher.
\newblock {\em {A}lgorithms for {M}ori {D}ream {S}paces}.
\newblock PhD thesis, Universit\"at T\"ubingen, 2014.

\bibitem{MacStu}
D.~Maclagan and B.~Sturmfels.
\newblock {\em Introduction to tropical geometry}, volume 161 of {\em Graduate
  Studies in Mathematics}.
\newblock American Mathematical Society, Providence, RI, 2015.

\bibitem{ReidBourbaki}
M.~Reid.
\newblock La correspondance de {M}c{K}ay.
\newblock {\em Ast\'erisque}, (276):53--72, 2002.
\newblock S\'eminaire Bourbaki, Vol.\ 1999/2000.

\bibitem{Tev}
J.~Tevelev.
\newblock Compactifications of subvarieties of tori.
\newblock {\em Amer. J. Math.}, 129(4):1087--1104, 2007.

\bibitem{Tev2}
J.~Tevelev.
\newblock On a question of {B}. {T}eissier.
\newblock {\em Collect. Math.}, 65(1):61--66, 2014.

\bibitem{Verbitsky_crepant}
M.~Verbitsky.
\newblock Holomorphic symplectic geometry and orbifold singularities.
\newblock {\em Asian J. Math.}, 4(3):553--563, 2000.

\bibitem{Ya}
R.~Yamagishi.
\newblock On smoothness of minimal models of quotient singularities by finite
  subgroups of {$SL_n(\mathbb{C})$}.
\newblock 2016.
\newblock Preprint. arXiv:1602.01572.

\end{thebibliography}
